\documentclass[a4paper]{amsart}
\usepackage{amsthm}
\usepackage{amsmath}
\usepackage{amssymb}
\usepackage{tikz}
\usepackage{color}
\usetikzlibrary{matrix}
\usepackage{dsfont}
\usepackage{graphicx} 
\usepackage[margin=1.2 in]{geometry}
\usepackage{stackrel}
\numberwithin{equation}{section}

\title{Quantum symmetry groups of noncommutative tori} 
\author{Micha{\l} Banacki}
\address{Institute of Theoretical Physics and Astrophysics, Faculty of Mathematics, Physics, and Informatics, University of Gda{\'n}sk, 80-308 Gda{\'n}sk, Poland}
\email{mibanfiz@gmail.com}
\author{Marcin Marciniak} 
\address{Institute of Theoretical Physics and Astrophysics, Faculty of Mathematics, Physics, and Informatics, University of Gda{\'n}sk, 80-308 Gda{\'n}sk, Poland}
\email{matmm@ug.edu.pl}

\theoremstyle{plain}
\newtheorem{thm}{Theorem}[section] 

\newtheorem{lm}[thm]{Lemma}
\newtheorem{pro}[thm]{Proposition}
\newtheorem{cor}[thm]{Corollary}

\theoremstyle{definition}
\newtheorem{defn}[thm]{Definition} 
\newtheorem{exmp}[thm]{Example} 

\theoremstyle{remark}
\newtheorem{m}[thm]{Remark}

\newcommand{\id}{\mathrm{id}}
\renewcommand{\t}{\tau}
\newcommand{\ti}{\tau^{-1}}
\newcommand{\s}{\sigma}
\newcommand{\si}{\sigma^{-1}}
\renewcommand{\r}{\rho}

\newcommand{\om}{\omega}
\newcommand{\vk}{{\vec{k}}}

\newcommand{\vr}{{\vec{r}}}
\newcommand{\ep}{\varepsilon}

\newcommand{\beq}{\begin{equation}}
\newcommand{\eeq}{\end{equation}}
\newcommand{\be}{\begin{eqnarray}}
\newcommand{\ee}{\end{eqnarray}}
\newcommand{\beg}{\begin{eqnarray*}}
\newcommand{\eeg}{\end{eqnarray*}}
\newcommand{\bC}{\mathbb{C}}
\newcommand{\cB}{\mathcal{B}}
\newcommand{\cC}{\mathcal{C}}
\newcommand{\cV}{\mathcal{V}}
\newcommand{\cK}{\mathcal{K}}
\newcommand{\bT}{\mathbb{T}}
\newcommand{\bZ}{\mathbb{Z}}

\newcommand{\bG}{\mathbb{G}}
\newcommand{\cA}{\mathcal{A}}
\newcommand{\jed}{\mathds{1}}
\newcommand{\te}{\theta}
\newcommand{\Ob}{\mathrm{Ob}}
\newcommand{\Mor}{\mathrm{Mor}}
\newcommand{\bR}{\mathbb{R}}
\newcommand{\la}{\lambda}

\begin{document}
\maketitle

\begin{abstract}
We discuss necessary conditions for a compact quantum group to act on the algebra of noncommutative $n$-torus $\bT_\theta^n$ in a filtration preserving way in the sense of Banica and Skalski. As a result, we construct a family of compact quantum groups $\bG_\theta=(A_\theta^n,\Delta)$ such that for each $\theta$, $\bG_\theta$ is the final object in the category of all compact quantum groups acting on $\bT_\theta^n$ in a filtration preserving way. We describe in details the structure of the C*-algebra $A_\theta^n$ and provide a concrete example of its representation in bounded operators. Moreover, we compute the Haar measure of $\bG_\te$. For $\theta=0$, the quantum group $\bG_0$ is nothing but the classical group $\bT^n\rtimes S_n$, where $S_n$ is the symmetric group. For general $\theta$, $\bG_\theta$ is still an extension of the classical group $\bT^n$ by the classical group $S_n$. It turns out that for $n=2$, the algebra $A_\theta^2$ coincides with the algebra of the quantum double-torus described by Hajac and Masuda. 
Using a variation of the little subgroup method we show that irreducible representations of $\bG_\theta$ are in one-to-one correspondence with irreducible representations of $\bT^n\rtimes S_n$.
\end{abstract}

\section{Introduction}
Classical notion of a group arises in a natural way in the context of symmetries of various mathematical structures. When the notion of a quantum group had been introduced then it became clear that this scheme could provide us with the new insight into the meaning of symmetry. It appears that the notion of a coaction of a compact quantum group on a C*-algebra (\cite{Pod95}) is an appropriate framework for description of nonclassical symmetries of noncommutative spaces.

Wang's considerations on quantum symmetry groups of finite spaces (\cite{Wan98}) was one of the first attempt to address this issue. It was considered also by other researchers in various contexts (see for instance \cite{FNW96,Mar98b}).  An interesting approach was presented by Goswami in \cite{Gos09}, where the concept of quantum isometry group of a noncommutative manifold was considered. In \cite{Bho,BhoG} several examples were provided. Motivated by this Banica and Skalski (\cite{BS13}) proposed a definition of a quantum symmetry group of C*-algebra equipped with an orthogonal filtration. 

The aim of this paper is to present an explicit construction of a quantum symmetry group $\bG_\te$ of the algebra of noncommutative torus $C(\bT_\theta^n)$ which acts in a filtration preserving way. We show that this quantum group appears in a natural way from considering some necessary conditions for such an action. Thus, it is a final object in the category described by Banica and Skalski. Moreover, we discuss in details the structure of its underlying C*-algebra $A_\theta^n$. In particular we describe a representation of $A_\theta^n$ as an algebra of bounded operators on some appropriate Hilbert space. We compute the Haar measure of $\bG_\te$. It appears that $\bG_\te$ is an analog of the classical semidirect product $\bT^n\rtimes S_n$. It is interesting that being a purely quantum group it is still an extension of the classical torus $\bT^n$ by the classical symmetric group $S_n$. We explore this property and we describe irreducible representations of $\bG_\te$ by means of a variation of little subgroups method. We show that irreducible representations of $\bG_\te$ are in one-to-one correspondence with irreducible representations of the classical group $\bT^n\rtimes S_n$. 

The paper is organized as follows. In Section 2 we provide some preliminary facts on compact quantum groups, their coactions on C*-algebras, and basic properties on noncommutative tori. Section 3 provides the reader with the construction of the C*-algebra $A_\theta^n$. In Chapter 4 we prove the existence of the quantum group $(A_\theta^n,\Delta)$, compute its Haar measure and describe its irreducible representations.

\section{Preliminaries}
\subsection{Compact quantum groups}
Let us begin with a short introduction to the theory of compact quantum groups developed by Woronowicz in \cite{Wor94,Wor87}. 
\begin{defn}[\cite{Wor94}]
\label{cqg}
A compact quantum group $\bG$ is a pair $\bG=(A, \Delta)$, where $A$ is a unital C*-algebra and $\Delta:A\rightarrow A\otimes A$ is a unital $^*$-homomorphism (comultiplication) such that
\begin{enumerate}
\item $\left(\Delta\otimes \id_A\right)\Delta=\left(\id_{A}\otimes \Delta\right)\Delta$, i.e. the following diagram
\begin{center}
\begin{tikzpicture}
  \matrix (m) [matrix of math nodes,row sep=3em,column sep=4em,minimum width=2em]
  {
     A & A\otimes A \\
     A\otimes A & A\otimes A\otimes A \\};
  \path[-stealth]
    (m-1-1) edge node [left] {$\Delta$} (m-2-1)
            edge node [above] {$\Delta$} (m-1-2)
    (m-2-1) edge node [below] {$\id_A\otimes \Delta$}
             (m-2-2)
    (m-1-2) edge node [right] {$\Delta\otimes \id_A$} (m-2-2);
\end{tikzpicture}
\end{center}
commutes,
\item subspaces $\Delta(A)(\mathds{1}\otimes A)$ and $\Delta(A)(A\otimes \mathds{1})$ are dense in $A\otimes A$.
\end{enumerate}
\end{defn} 
In particular, the above definition is satisfied by compact matrix quantum groups. 
\begin{defn}[\cite{Wor87}]
A compact matrix quantum group is a pair $(A,u)$ which consists of a unital C*-algebra $A$ and a matrix $u\in M_n(A)$ satisfying the following conditions
\begin{enumerate}
\item matrix entries $u_{ik}$ generate dense $^*$-subalgebra $\mathcal{A}$ in $A$,
\item there exists a unital $^*$-homomorphism (comultiplication) $\Delta:A\rightarrow A\otimes A$ such that 
$$ 
\Delta(u_{ik})=\sum_{j=1}^n u_{ij}\otimes u_{jk},\qquad i,k=1,\ldots, n,
$$ 
\item there exists a linear antimultiplicative map (coinverse) $\kappa:\mathcal{A}\rightarrow \mathcal{A}$ such that 
$$ 
\kappa(\kappa(a^*)^*)=a, \qquad a\in\cA,
$$ 
and 
\beg 
\sum_{k=1}^n u_{ik}\kappa(u_{kj})&=&\delta_{ij}\mathds{1}_A, \\
\sum_{k=1}^n \kappa(u_{ik})u_{kj}&=&\delta_{ij}\mathds{1}_A
\eeg 
for $i,j=1,\ldots, n$.
\end{enumerate}
\end{defn} 
Analogously to the case of classical compact topological groups, we introduce the notion of a finite-dimensional representation. Remind that according to 'leg notation', if $t=\sum_ix_i\otimes y_i\in X\otimes Y$ for some unital algebras $X,Y$ and $Z$ is also a unital algebra, then elements $t_{12}\in X\otimes Y\otimes Z$ and $t_{13}\in X\otimes Z\otimes Y$ are defined as $t_{12}=\sum_ix_i\otimes y_i\otimes \jed_Z$ and $t_{13}=\sum_ix_i\otimes\jed_Z\otimes y_i$.
\begin{defn}
Let $H$ be a finite dimensional Hilbert space. A 
representation of a compact quantum group $\bG=(A,\Delta)$ acting on $H$ is an element $v\in B(H)\otimes A$ such that
\begin{equation}
\label{eqrep}
\left(\id_A\otimes \Delta\right)v=v_{12}v_{13}.
\end{equation}
\end{defn}

Let us observe that for a rerpresentation $v=\sum_ib_i\otimes a_i\in B(H)\otimes A$, one can consider a uniquely determined linear map $\tilde{v}:H\rightarrow H\otimes A$ defined by
$$ 
\tilde{v}(x)=\sum_ib_ix\otimes a_i,\qquad x\in H.
$$ 
Then the condition \eqref{eqrep} means that the following diagram commutes
\begin{center}
\begin{tikzpicture}
  \matrix (m) [matrix of math nodes,row sep=3em,column sep=4em,minimum width=2em]
  {
     H & H\otimes A \\
     H\otimes A & H\otimes A\otimes A \\};
  \path[-stealth]
    (m-1-1) edge node [left] {$\tilde{v}$} (m-2-1)
            edge node [above] {$\tilde{v}$} (m-1-2)
    (m-2-1) edge node [below] {$\id_H\otimes \Delta$}
             (m-2-2)
    (m-1-2) edge node [right] {$\tilde{v}\otimes \id_A$} (m-2-2);
\end{tikzpicture}
\end{center}
By choosing particular orthonormal basis $e_{1},\ldots,  e_n$ in $H$ we can define $v:H\rightarrow H\otimes A$ by the set of $n^2$ 'matrix coefficients' $v_{ij}\in A$ satisfying
$
v(e_j)=\sum_{i=1}^ne_i\otimes v_{ij}.
$ 
Now, the condition 
\eqref{eqrep} is equivalent to
$$ 
\Delta(v_{ij})=\sum_{k=1}^nv_{ik}\otimes v_{kj}, \qquad i,j=1,2,\ldots,n.
$$ 

We say that representation $v$ is unitary if there exists an orthonormal basis $e_{1},\ldots,  e_n$ in $H$ such that 'matrix coefficients' $v_{ij}\in A$ form a unitary matrix in $M_n(A)$.

Let us consider two representations $v:K\rightarrow K\otimes A$ and $w:L\rightarrow L\otimes A$. A linear map $S:K\rightarrow L$ such that the following diagram 
\begin{center}
\begin{tikzpicture}
  \matrix (m) [matrix of math nodes,row sep=3em,column sep=4em,minimum width=2em]
  {
     K & K\otimes A \\
     L & L\otimes A \\};
  \path[-stealth]
    (m-1-1) edge node [left] {$S$} (m-2-1)
            edge node [above] {$v$} (m-1-2)
    (m-2-1) edge node [below] {$w$}
             (m-2-2)
    (m-1-2) edge node [right] {$S\otimes \id_A$} (m-2-2);
\end{tikzpicture}
\end{center}
commutes is know as an intertwining operator. The set of intertwining operators will be denoted by $\Mor(v,w)$. We say that $v$ is equivalent to $w$ if there exists an invertible element in $\Mor(v,w)$. Representation $v$ is called irreducible if and only if $\Mor(v,v)=\left\{\lambda\mathds{1}_{B(K)}:\lambda\in \mathds{C}\right\}$. Each representation is equivalent to a unitary one.\newline

Finally, let us remind the notion of the Haar measure  
\begin{thm}[\cite{Wor94}]
 Let $\bG=(A,\Delta)$ be a compact quantum group. There exists a unique state $h$ on $A$ 
such that
\begin{equation}
(\id_A\otimes h)\Delta=h(\cdot)\mathds{1}_A=(h\otimes \id_A)\Delta.
\end{equation}
\end{thm}
\noindent We say that $h$ is the Haar measure on $\bG=(A,\Delta)$.

\subsection{Noncommutative $n$-torus}

Idea of noncommutative $n$-torus was first considered by Rieffel \cite{Rie90} as a generalization of the irrational rotational algebra \cite{Rie81}. Since that moment noncommutative torus has become broadly discussed subject as one of the simplest yet still nontrivial example of a noncommutative topological space (which can be also equipped with a noncommutative smooth structure in the sense of Connes spetral triples formalism \cite{Con94,Con13}).

\begin{defn} Noncommutative $n$-torus is a universal C*-algebra $C(\mathds{T}_\theta^n)$ generated by the set of unitary elements $\left\{x_i:i=1,2\ldots,n\right\}$ satisfying 
$$ 
x_ix_j=\omega_{ij}x_jx_i 
$$ 
for all $i,j=1,2\ldots,n$, where $\omega_{ij}=e^{2\pi i\theta_{ij}}$ and $\theta_{ij}$ are entries of a skew-symmetric matrix $\theta\in M_n(\mathbb{R})$.
\end{defn}
In a specific case when $\theta=0$, i.e. all $\omega_{ij}$ are equal to $1$, $C(\mathds{T}_0^n)$ is a universal C*-algebra generated by $n$ commuting unitaries $\left\{x_i:i=1,2\ldots,n\right\}$, therefore it is isomorphic with the algebra $C(\mathds{T}^n)$ of continuous functions on classical $n$-torus. Hence, noncommutative torus can be considered as a deformation of the commutative C*-algebra $C(\mathds{T}^n)$. In a language of noncommutative topology one can treat $C(\mathds{T}_\theta^n)$ as a dual object to the 'abstract' space $\mathds{T}^n_\theta$.

By $\mathrm{Poly}(\mathds{T}_\theta^n)$ we denote a dense $^*$-subalgebra of $C(\mathds{T}_\theta^n)$ which consists of elements of the following form
$$ 
\mathrm{Poly}(\mathds{T}_\theta^n)=\left\{x=\sum_{\vec{r}\in \mathds{Z}^n}a_{\vec{r}}x^{\vec{r}}:\textrm{almost all }a_{\vec{r}}\in \mathbb{C}\textrm{ are equal to } 0 \right\}
$$ 
where
$$ 
x^{\vr}=x_1^{r_1}x_2^{r_2}\ldots x_n^{r_n}
$$ 
for any $\vr=(r_1,r_2\ldots r_n)\in \mathds{Z}^n$. Similarly to the previous remark, $\mathrm{Poly}(\mathds{T}_0^n)$ coincides with the $^*$-algebra of polynomial functions on $\mathds{T}^n$, when monomial function $x_i:\mathds{T}^n\rightarrow \mathbb{C}$ is given by $x_i(t_1,t_2\ldots t_n)=t_i$.

There is a natural action $\gamma$ of $\mathds{T}^n$ on $C(\mathds{T}_\theta^n)$ given on generators as
\begin{equation}
\label{gama}
\gamma_t(x_i)=t_ix_i
\end{equation}where $t=(t_1,t_2\ldots t_n)\in \mathds{T}^n$. Strating from this definition one can consider an averaging over orbits $E:\mathrm{Poly}(\mathds{T}_\theta^n)\rightarrow \mathrm{Poly}(\mathds{T}_\theta^n)$ described by
$$ 
E(x)=\int_{\mathds{T}^n}{\gamma_t(x)d\mu(t)}=a_{\vec{0}}\mathds{1}=\phi(x)\mathds{1}
$$ 
where $\mu$ denotes the Haar measure on $\mathds{T}^n$. From that one can extends $\phi$ to the faithful trace define on the whole $C(\mathds{T}_\theta^n)$ \cite{CXZ13}.

\subsection{Quantum symmetry groups}
\label{subs_qsg}
The following definition due to Podle\'s \cite{Pod95} give us a generalization of a well known concept of strongly continuous action of a compact group on some unital C*-algebra \cite{Ped79}.
\begin{defn}\label{act}
A coaction of a compact quantum group $\bG=(A, \Delta)$ on a unital C*-algebra $B$ is a unital $^*$-homomorphism $\alpha:B\rightarrow B\otimes A$ such that
\begin{enumerate}
\item $\left(\alpha\otimes \id_A\right)\alpha=\left(\id_{B}\otimes \Delta\right)\alpha$, i.e. the following diagram
\begin{center}
\begin{tikzpicture}
  \matrix (m) [matrix of math nodes,row sep=3em,column sep=4em,minimum width=2em]
  {
     B & B\otimes A \\
     B\otimes A & B\otimes A\otimes A\\};
  \path[-stealth]
    (m-1-1) edge node [left] {$\alpha$} (m-2-1)
            edge node [above] {$\alpha$} (m-1-2)
    (m-2-1) edge node [below] {$\id_B\otimes \Delta$}
             (m-2-2)
    (m-1-2) edge node [right] {$\alpha\otimes \id_A$} (m-2-2);
\end{tikzpicture}
\end{center}
\noindent is commutative,
\item subspace $\alpha\left(B\right)\left(\mathds{1}_B\otimes A\right)$ is dense in $B\otimes A$.
\end{enumerate}
\end{defn} 

In the classical setting symmetry group (automorphism group) of a given space $X$ is described by the group of all transformations of $X$ which preserve its inner structure. From the category-theoretic point of view, the notion of symmetry group can be equivalently defined as a final (universal) object in an appropriate category of transformations acting on a given $X$. Such approach was used by Wang in his discussion on quantum automorphism groups of finite spaces \cite{Wan98}. Similar scheme was also introduced by Banica and Skalski in their definition of quantum symmetry groups preserving orthogonal filtrations \cite{BS13}. 

Before we pass to this construction we shall remind the notion of an orthogonal filtration. Suppose that there is a unital C*-algebra $B$ and the set of indexes $I$ with one distinct element $0\in I$.

\begin{defn}[\cite{BS13}]
An orthogonal filtration on $B$ is a pair $\cV=\left\{\varphi,(V_i)_{i\in I}\right\}$ which consists of a faithful state $\varphi: B\rightarrow \mathbb{C}$  and a family $(V_i)_{i\in I}$ of finite-dimensional subspaces of $B$ such that
\begin{enumerate}
\item $V_0=\mathbb{C}\mathds{1}$,
\item $\varphi(x^*y)=0$ for $x\in V_i$, $y\in V_j$, where $i,j\in I$, $i\neq j$,
\item 
$\mathrm{span}\left(\bigcup_{i\in I}V_i\right)$ is a dense 
$^*$-subalgebra $\cB$ of $B$.
\end{enumerate}
\end{defn}
\begin{exmp}
\label{exfol}
Let $B=C(\bT_\te^n)$ and let $\phi$ be the faithful trace on $C(\bT_\te^n)$.
We will consider the following orthogonal filtration on $C(\bT_\te^n)$.
\begin{enumerate}
\item $V_0=\mathbb{C}\mathds{1}$,
\item $V_{p,q}=\mathrm{span}\{x_{i_1}x_{i_2}\ldots x_{i_p}x^*_{j_q}\ldots x^*_{j_1}:\,\mbox{$i_k,j_l=1,\ldots,n$, $i_k\neq j_l$, $k=1,\ldots,p$, $l=1,\ldots,q$}\}$, 
\end{enumerate}
i.e. $V_{p,q}$ is a linear span of irreducible words consisting of $p$ generators $x_i$ and $q$ conjugate generators $x_j^*$.
\end{exmp}

\begin{defn} Let $B$ be a unital C*-algebra equipped with an orthogonal filtration $\cV=\left\{\varphi,(V_i)_{i\in I}\right\}$. We say that a compact quantum group $\bG=(A_\bG,\Delta_\bG)$ acts on $B$ in a filtration preserving way if there is a coaction $\alpha:B\rightarrow B\otimes A_\bG$ such that 
\begin{equation}
\alpha(V_i)\subset V_i\otimes A_\bG
\end{equation}
for all $i\in I$, where $V_i\otimes A_\bG$ denotes here an algebraic tensor product.
\end{defn}

By $\mathcal{C}_{B,\cV}$ (or $\cC_B$ when the state and filtrafion are fixed) we denote the category of all compact quantum groups acting on $B$ in a filtration preserving way. 
Objects $\Ob\left(\mathcal{C}_{B,\cV}\right)$ in this category are pairs $(\bG,\alpha)$ where $\bG$ is a compact quantum group and $\alpha$ is a filtration preserving coaction of $\bG$ on $B$. 
The morphisms $\Mor\left(\mathcal{C}_{B,\cV}\right)$ are compact quantum groups morphisms which are compatible with appropriate coactions, i.e. $\pi\in \Mor((\bG_1,\alpha_1),(\bG_2,\alpha_2))$ if $\pi: C(\bG_2)_u\to C(\bG_1)_u$ is a unital *-homomorphism such that
\beq
\label{univ1}
(\pi\otimes\pi)\circ\Delta_{\bG_2}=\Delta_{\bG_1}\circ\pi,
\eeq
and
\beq
\label{univ2}
(\id_B\otimes\pi)\circ\alpha_2\big|_\cB=\alpha_1\big|_\cB.
\eeq
(Here $C(\bG)_u$ denotes the universal version of $A_\bG$. 
See \cite{BS13,BCT05} for details).

We say that $(\bG_u,\alpha_u)\in \Ob\left(\mathcal{C}_{B,\cV}\right)$ is a quantum symmetry group of $(B,\cV)$ if 
$(\bG_u,\alpha_u)$ is a final object in this category, i.e for each $(\bG,\alpha)\in \Ob\left(\mathcal{C}_{B,\cV}\right)$ there is a unique morphism from $(\bG,\alpha)$ to  $(\bG_u,\alpha_u)$.

\begin{m}
The existence of the final object in $\mathcal{C}_{B,\cV}$ 
was proved in 
\cite{BS13}. The final object is unique up to isomorphism.
\end{m}

Let $\bG=(A,\Delta)$ be a compact quantum group which coacts on a unital C*-algebra $B$ by unital $^*$-homomorphism $\alpha$. We say that the coaction is 
faithful 
if there is no proper Woronowicz Hopf C*-subalgebra $\widetilde{A}\subset A$ such that $\alpha$ is a coaction of $(\widetilde{A},\Delta)$ on $B$.
If $(\bG_u,\alpha_u)$ is the final object in $\mathcal{C}_{B,\cV}$, then $\alpha_u$ is faithful \cite{Wan98}.

\section{Construction of the quantum group algebra}
\subsection{Necessary conditions}
Let 
$\bG=(A,\Delta)$ be a compact quantum group. Assume that $\alpha$ is a coaction of $\bG$ on $C(\bT_\te^n)$ which 
preserves the filtration described in Example \ref{exfol}.  
The aim of this subsection is to describe  
necessary conditions for $(A,\Delta)$ to be the final object in the category $\cC_{C(\bT_\te^n)}$. 
Let us start with the observation that 
the assumption on $\alpha$ implies $\alpha(V_{1,0})\subset V_{1,0}\otimes A$. Since $V_{1,0}=\mathrm{span}\{x_1,\ldots,x_n\}$, there are unique elements $u_{ik}\in A$, where $i,k=1,\ldots,n$, such that  
\begin{equation}
\label{eq5}
\alpha(x_k)=\sum_{i=1}^n x_i\otimes u_{ik},\qquad k=1,\ldots,n.
\end{equation}
Definition \ref{act} and (\ref{eq5}) imply that 
\begin{equation}
\label{del}
\Delta(u_{ik})=\sum_j u_{ij}\otimes u_{jk},\qquad i,k=1,\ldots,n.
\end{equation}
Therefore $u=(u_{ik})\in M_n(A)=B(\bC^n)\otimes A$ is a representation of $\bG$ on $\bC^n$. 
\begin{pro}
\label{profil}
If $\alpha$ preserves the filtration on the algebra $C(\bT_\te^n)$, then elements $u_{ik}$ satisfy the following relations:
\be
&u_{ik}u_{jl}+\omega_{ji}u_{jk}u_{il}= \omega_{kl}u_{il}u_{jk}+\omega_{ji}\omega_{kl}u_{jl}u_{ik},& 
\label{e1.1}\\
&\sum\limits_{i=1}^n u_{ik}^{}u_{il}^*=\delta_{kl}\mathds{1},& 
\label{e1.2a}\\
&\sum\limits_{i=1}^n u_{il}^*u_{ik}^{}=\delta_{kl}\mathds{1},&  
\label{e1.2b} \\
& u_{jk}^{}u_{ik}^*=0,& 
i\neq j, \label{e1.4a} \\
& u_{ik}^*u_{jk}=0 & i\neq j, \label{e1.4b}
\ee
where $i,j,k,l=1,\ldots,n$.
\end{pro}
\begin{proof}
Since $\alpha$ is a *-homomorphism, we have
\begin{equation}
\label{w1}
 \alpha(x_kx_l)=\omega_{kl}\alpha(x_lx_k)
\end{equation}
\begin{equation}\label{w2}
\alpha(x_kx_k^*)=\mathds{1}\otimes\mathds{1}=\alpha(x_k^*x_k)
\end{equation}
for every $k,l=1,\ldots,n$.
Let us observe that
\beg
\alpha\left(x_kx_l\right)&=&\alpha\left(x_k\right)\alpha\left(x_l\right)=\sum_{i,j}x_ix_j\otimes u_{ik}u_{jl}
\\&=&
 \sum_{i}x_i^2\otimes u_{ik}u_{il}+\sum_{i<j}x_ix_j\otimes u_{ik}u_{jl}+\sum_{i>j}x_ix_j\otimes u_{ik}u_{jl}\\
&=&
 \sum_{i}x_i^2\otimes u_{ik}u_{il}+\sum_{i<j}x_ix_j\otimes u_{ik}u_{jl}+\sum_{i<j}x_jx_i\otimes u_{jk}u_{il}\\
&=& \sum_{i}x_i^2\otimes u_{ik}u_{il}+\sum_{i<j}x_ix_j\otimes u_{ik}u_{jl}+\sum_{i<j}\omega_{ji}x_ix_j\otimes u_{jk}u_{il} \\
&=& \sum_{i}x_i^2\otimes u_{ik}u_{il}+\sum_{i<j}x_ix_j\otimes (u_{ik}u_{jl}+\omega_{ji} u_{jk}u_{il}) 
\eeg 
and, consequently
$$ 
\alpha\left(x_lx_k\right)=
\sum_{i}x_i^2\otimes  
u_{il}u_{ik}+\sum_{i<j}x_ix_j\otimes (
u_{il}u_{jk}+ 
\omega_{ji} u_{jl}u_{ik}) .
$$ 
Since the system $\{x_ix_j:\,i\leq j\}$ is linearly independent, \eqref{w1} leads to \eqref{e1.1}.
Further, we have
\beg
\alpha\left(x_kx_k^*\right)&=&\alpha\left(x_k\right)\alpha\left(x_k^*\right)= \sum_{i,j}x_i^{}x_j^*\otimes u_{ik}u_{jk}^*
=
 \sum_{i}x_i^{}x_i^*\otimes u_{ik}u_{ik}^*+\sum_{i\neq j}x_i^{}x_j^*\otimes u_{ik}u_{jk}^*\\
&=&
 \mathds{1}\otimes \sum_{i}u_{ik}u_{ik}^*+\sum_{i\neq j}x_i^{}x_j^*\otimes u_{ik}u_{jk}^*
\eeg 
and similarly
$$
\alpha\left(x_k^*x_k\right)=
 \mathds{1}\otimes \sum_{i}u_{ik}^*u_{ik}^{}+\sum_{i\neq j}x_i^*x_j^{}\otimes u_{ik}^*u_{jk}.
$$ 
Using \eqref{w2} and linear independence of the set $\{\jed\}\cup\{x_i^{}x_j^*:\,i\neq j\}$ we derive 
\begin{equation}
\label{e2aa}
\sum_i u_{ik}^{}u_{ik}^*=\mathds{1}, \qquad k=1,2\ldots,n
\end{equation}
and \eqref{e1.4a}. 
Analogously, \eqref{w2} and linear independence of the set $\{\jed\}\cup\{x_i^{*}x_j^{}:\,i\neq j\}$ imply 
\beq
\label{e2ba}
\sum_i u_{ik}^{*}u_{ik}^{}=\mathds{1}, \qquad k=1,2\ldots,n
\eeq
and \eqref{e1.4b}.
Finally, we make use of the inclusion 
$\alpha (V_{1,1})\subset V_{1,1}\otimes A$. Since 
$$ 
\alpha(x_kx_l^*)=\sum_{i,j} x_ix_j^*\otimes u_{ik}u^*_{jl}=\mathds{1}\otimes \sum_{i}u_{ik}u^*_{il}+\sum_{i\neq j} x_ix_j^*\otimes u_{ik}u^*_{jl}
$$ 
and
$$ 
\alpha(x_k^*x_l)=\sum_{i,j} x_i^*x_j\otimes u_{ik}^*u_{jl}=\mathds{1}\otimes \sum_{i}u^*_{ik}u_{il}+\sum_{i\neq j} x_i^*x_j\otimes u_{ik}^*u_{jl}
$$ 
for $k\neq l$, we derive 
\begin{equation}
\label{e2ab}
\sum_{i}u_{ik}u^*_{il}=0
\end{equation}
and
\begin{equation}
\label{e2bb}
\sum_{i}u^*_{ik}u_{il}=0.
\end{equation}
Combination of \eqref{e2aa} and \eqref{e2ab} leads to \eqref{e1.2a}, while \eqref{e2ba} and \eqref{e2bb} give \eqref{e1.2b}.
\end{proof}

\subsection{Construction of a representation}
The aim of this subsection is to describe explicitly a conecrete Hilbert space $H$ and operators $U_{ik}\in B(H)$, $i,k=1,\ldots,n$, which satisfy relations listed in Proposition \ref{profil}.
To this end let us define 
$H=\bigoplus_{\s\in S_n}H_\s$, where 
$H_\s=\ell^2(\mathbb{Z})^{\otimes{n \choose 2}}$ for every $\s\in S_n$, where $S_n$ denotes the symmetric group.
Let 
$$
\Lambda=\{(\lambda_1,\lambda_2):\,\lambda_1,\lambda_2=1,2,\ldots,n,\; \lambda_1<\lambda_2\}.
$$ 
Obviously, $\#\Lambda={n \choose 2}$. Tensor factors in $H_\s=\ell^2(\bZ)^{\otimes {n\choose 2}}$ will be labeled by elements of $\Lambda$ and by $\left\{e_{\s,m}^{\lambda}:\, m\in \mathbb{Z}\right\}$ we denote the standard orthonormal basis in the copy of  $\ell^2\left(\mathbb{Z}\right)$ labeled by $\lambda\in\Lambda$. By $V$ we denote the set of all functions $v:\Lambda\rightarrow \mathbb{Z}$. For every $\sigma\in S_n$ and $v\in V$ we 
define vectors $\ep_{\s,v}\in H_{\sigma}= \ell^2(\mathbb{Z})^{\otimes{n \choose 2}}$ by
$$ 
\ep_{\s,v}=\bigotimes\limits_{\lambda\in \Lambda}e^{\lambda}_{\s,v(\lambda)}.
$$ 
The set $\left\{\ep_{\s,v}: v\in V\right\}$ is an orthonormal basis in each space $H_{\sigma}$. 
Define action of each operator $U_{ik}$ on the basis $\{\ep_{\s,v}:\,\s\in S_n,\,v\in V\}$ 
by
$$ 
U_{ik}\ep_{\s,v}= 
\begin{cases} \bigotimes\limits_{\lambda\in \Lambda}U_{ik}^{\lambda}e^{\lambda}_{\s,v(\lambda)} &\mbox{if } \sigma(k)=i \\
0 & \mbox{if } \sigma(k)\neq i\end{cases}, 
$$ 
where
$$ 
U_{ik}^{\lambda}e_{\s,m}^\lambda= \begin{cases} e^{\lambda}_{\s,m} &\mbox{if } k\notin \lambda ,\\
(\overline{\omega_{i,\sigma(\lambda_2)}}\, \omega_{k,\lambda_2})^{m}\,e^{\lambda}_{\s,m} & \mbox{if } k=\lambda_1 ,\\
e^{\lambda}_{\s,m+1}& \mbox{if } k=\lambda_2. \end{cases}
$$ 
Then, the adjoint $U_{ik}^*$ is given by  
$$ 
U_{ik}^*\ep_{\s,v}= 
\begin{cases} \bigotimes\limits_{\lambda\in \Lambda}U_{ik}^{*\lambda}e^{\lambda}_{\s,v(\lambda)} &\mbox{if } \sigma(k)=i , \\
0 & \mbox{if } \sigma(k)\neq i ,\end{cases}
$$ 
where
$$ 
U_{ik}^{*\lambda}e^{\lambda}_{\s,m}= \begin{cases} e^{\lambda}_{\s,m} &\mbox{if } k\notin \lambda ,\\
(\omega_{i,\sigma(\lambda_2)}\,\overline{\omega_{k,\lambda_2}})^{m}\,e^{\lambda}_{\s,m} & \mbox{if } k=\lambda_1,\\
e^{\lambda}_{\s,m-1}& \mbox{if } k=\lambda_2.\end{cases}  
$$ 
\begin{pro}
\label{rep}
Operators $U_{ik}$ satisfy relations listed in Proposition \ref{profil}.
\end{pro}
\begin{proof}
It follows from the above definitions that 
$$
U_{ik}^\#(H_\s)= 
\begin{cases}
H_\s & \mbox{if $\s(k)=i$,} \\
\{0\} & \mbox{if $\s(k)\neq i$,}
\end{cases}
$$
where $a^\#$ means either $a$ or $a^*$. 
If $i,j,k,l$ are such that $i=j$, $k\neq l$ or $i\neq j$, $k=l$, then there is no permutation $\sigma$ such that $\sigma(k)=i$ and $\sigma(l)=j$. Hence, for any $\sigma$, either $U_{ik}^\#(H_\s)=\{0\}$ or $U_{jl}^\#(H_\s)=\{0\}$. Therefore,
\begin{equation}
\label{e1.13a}
U_{ik}^\#U_{jk}^\#=0,\qquad i\neq j 
\end{equation} 
and
\beq
\label{e1.13b}
U_{ik}^\#U_{il}^\#=0,\qquad k\neq l.
\eeq

Now, assume $i\neq j$ and $k\neq l$. Let $\sigma$ be such that $\sigma(k)=i$ and $\sigma(l)=j$. 
Firstly, observe that 
\beq
\label{tru}
U_{ik}^{\#\lambda} U_{jl}^{\#\lambda}e_{\s,m}^\la = U_{jl}^{\#\lambda} U_{ik}^{\#\lambda}e_{\s,m}^\la
\eeq 
if  $k\notin\lambda$ or $l\notin\lambda$. Further, assume that $k<l$ and $\lambda_0=(k,l)$. Then
$$ 
U_{ik}^{\lambda_0} U_{jl}^{\lambda_0} e^{\lambda_0}_{\s,m}= (\omega_{ji}\omega_{kl})^{m+1}\,e^{\lambda_0}_{\s,m+1}
$$ 
$$ 
U_{jl}^{\lambda_0} U_{ik}^{\lambda_0} e^{\lambda_0}_{\s,m}=
(\omega_{ji}\omega_{kl})^{m}\,e^{\lambda_0}_{\s,m+1}
$$ 
\beg
U_{ik}U_{jl}\ep_{\s,v}&=&\bigotimes\limits_\la U_{ik}^{\la}U_{jl}^{\la}e_{\s,v(\la)}^\la = U_{ik}^{\la_0}U_{jl}^{\la_0}e_{\s,v(\la_0)}^{\la_0}\otimes \bigotimes\limits_{\la\neq\la_0} U_{ik}^{\la}U_{jl}^{\la}e_{\s,v(\la)}^\la\\
&=&(\omega_{ji}\omega_{kl})^{v(\la_0)+1}\,e^{\lambda_0}_{\s,v(\la_0)+1} \otimes\bigotimes\limits_{\la\neq\la_0}
U_{ik}^{\la}U_{jl}^{\la}
 e_{\s,v(\la)}^\la
\eeg
and similarly
$$
U_{jl}U_{ik}\ep_{\s,v}=(\omega_{ji}\omega_{kl})^{v(\la_0)}\,e^{\lambda_0}_{\s,v(\la_0)+1} \otimes\bigotimes\limits_{\la\neq\la_0}
U_{jl}^{\la}U_{ik}^{\la}
e_{\s,v(\la)}^\la .
$$

So, bearing in mind \eqref{tru}, we get
$$
U_{ik}U_{jl}\ep_{\s,v}=\omega_{ji}\omega_{kl}U_{jl}U_{ik}\ep_{\s,v}.
$$

Since $\s(k)\neq j$ and $\s(l)\neq i$, we have also
$$ 
U_{jk} U_{il}\ep_{\s,v}=U_{il}U_{jk}\ep_{\s,v}=0
$$ 
Thus, relation \eqref{e1.1} is satisfied on the subspace $H_\s$. Similarly, we check that it holds on $H_\s$, where $\s$ is such that $\s(k)=j$ and $\s(l)=i$. If $\s(k)\notin\{i,j\}$ or $\s(l)\notin\{i,j\}$, then all four terms in \eqref{e1.1} vanish on $H_\s$. Hence, \eqref{e1.1} is satisfied on the whole space $H$.
Further, if $i,j,k$ are such that $i\neq j$, then due to \eqref{e1.13a}
$$ 
U_{ik}U_{jk}^*=0.
$$ 
$$ 
U_{ik}^*U_{jk}=0.
$$ 
Therefore, relations \eqref{e1.4a} and \eqref{e1.4b} are satisfied.
Finally, let us observe that 
$$ 
U_{ik}U_{ik}^*\ep_{\s,v}= \begin{cases} \ep_{\s,v} &\mbox{if } \sigma(k)=i \\
0 & \mbox{if } \sigma(k)\neq i\end{cases}
$$ 
and
$$ 
\label{e1.23}
U_{ik}^*U_{ik}\ep_{\s,v}= \begin{cases} \ep_{\s,v} &\mbox{if } \sigma(k)=i \\
0 & \mbox{if } \sigma(k)\neq i\end{cases}.
$$ 
Since for every $k$ and $\sigma$ there is exactly one index $i$ such that $\sigma(k)=i$,
$$ 
\sum_i U_{ik}U_{ik}^*\ep_{\s,v}=\ep_{\s,v}
$$ 
and
$$ 
\sum_i U_{ik}^*U_{ik}\ep_{\s,v}=\ep_{\s,v}.
$$ 
Thus, relations \eqref{e1.2a} and \eqref{e1.2b} are also satisfied.
\end{proof}
\begin{cor}
There exists a universal C*-algebra $A_u$ generated by $u_{ik}$, $i,k=1,\ldots,n$, subject to relations listed in Proposition \ref{profil}. 
\end{cor}
\begin{proof}
In order to show that $A_u$ is well defined one need to construct at least one representation of the relations \eqref{e1.1} -- \eqref{e1.4b} in bounded operators and to show that 
$$
\sup\{\Vert\pi(u_{ik})\Vert:\,\mbox{$\pi$ is a representation of relations \eqref{e1.1} -- \eqref{e1.4b}}\} <\infty
$$
for every $i,k$ (cf. \cite{Bla}).
The first task was done in Proposition \ref{rep}.
Now, assume that for some Hilbert space $H$, operators $\pi(u_{ik})$ satisfy relations listed in Proposition \ref{profil}. It follows from relation \eqref{e1.2b} that  for every $x\in H$, 
$$ 
\left\|x\right\|^2=\left\langle x,x \right\rangle=\sum_i\left\langle x,\pi\left(u^*_{ik}u_{ik}\right)x \right\rangle = \sum_{i}\left\|\pi\left(u_{ik}\right)x\right\|^2,
$$  
hence
$\left\|\pi\left(u_{ik}\right)\right\|\leq 1$,
and
$\sup_{\pi}\left\|\pi(u_{ik})\right\|\leq 1 $. 
\end{proof}

\subsection{Structure of the algebra $A_u$}
For $\te=(\te_{ij})$ being a real skew-symmetric $n\times n$ matrix, we consider a \textit{quantum multitorus} C*-algebra $A_\te^n$ defined as  
$$ 
A_\te^n=\bigoplus_{\s\in S_n}C\left(\bT_{\te^{(\s)}}^n\right),
$$ 
where for $\sigma\in S_n$, $\te^{(\s)}=(\te_{ij}^{(\s)})\in M_n(\bR)$ is a skew-symmetric matrix defined by 
$$ 
\te_{ij}^{(\s)}=\te_{ji}+\te_{\si(i),\si(j)}.
$$ 
For any $\s\in S_n$, let $x_{\s,1},\ldots,x_{\s,n}$ be the system of generators of $C(\bT_{\te^{(\s)}}^n)$. The algebra $A_\te^n$ is generated by all $x_{\s,i}$, $i=1,\ldots,n$, $\s\in S_n$. The elements satisfy the following
relations
\beq
\label{mult1}
x_{\s,i}x_{\t,j}=\delta_{\s,\t}\om_{ij}^{(\s)}x_{\t,j}x_{\s,i},
\eeq
\beq
\label{mult2}
x_{\s,i}^{} x_{\t,j}^*=\delta_{\s,\t}\om_{ji}^{(\s)} x_{\t,j}^* x_{\s,i}^{},
\eeq
\beq
\label{mult3}
\sum_{\s\in S_n}x_{\s,i}^{} x_{\s,i}^*=\jed=\sum_{\s\in S_n} x_{\s,i}^*x_{\s,i}^{},
\eeq
where $\om_{ij}^{(\s)}=e^{\mathrm{2\pi i}\te_{ij}^{(\s)}}=\om_{j,i}\om_{\si(i),\si(j)}$.
\begin{m}
\label{m:mult}
It can be shown by standard arguments (\cite{Bla}) that $A_\te^n$ is isomorphic to the universal C*-algebra generated by elements $x_{\s,i}$ subject to relations \eqref{mult1} -- \eqref{mult3}. 
\end{m}

Our goal is to show the following.
\begin{thm}
\label{iso}
The universal C*-algebras $A_u$ is isomorphic to $A_\te^n$.
\end{thm}
Before the proof of the theorem we formulate some necessary propositions about operators satisfying relations listed in Proposition \ref{profil}. We start with the following two lemmas.
\begin{lm}
\label{er} 
Let $p_1,\ldots,p_n$ 
be orthogonal projections on some Hilbert space $H$. 
\begin{enumerate}
\item
If $\sum_{i=1}^n p_i=n\mathds{1}_{B(H)}$, then $p_i=\mathds{1}_{B(H)}$ for each $i$.
\item
If $\sum_{i=1}^n p_i=\mathds{1}_{B(H)}$, then $p_ip_j=0$ for every $i,j$ such that $i\neq j$.
\end{enumerate}
\end{lm}
\begin{proof} Routine. \end{proof}
\begin{lm}
\label{lem2}
Let $H$ be a Hilbert space. 
\begin{enumerate}
\item
If $S_1,S_2\subset H$ are two subsets, then $S_1^\perp \cap S_2^\perp=(S_1+S_2)^\perp$.
\item
If $K_1,K_2\subset H$ are closed subspaces, then $K_1^\perp+ K_2^\perp=(K_1\cap K_2)^\perp$.
\end{enumerate}
\end{lm}
\begin{proof} Elementary. \end{proof}

Now, we assume that $H$ is some Hilbert space, and $U_{ik}\in B(H)$ are arbitrary operators satisfying relations \eqref{e1.1} -- \eqref{e1.4b}. Let $P_{ik}=U_{ik}^*U_{ik}^{}$ and $Q_{ik}=U_{ik}^{}U_{ik}^*$.
\begin{pro} 
\label{prop}
Operators $U_{ik}^{}$ and $U_{ik}^*$ are partial isometries.
\end{pro}
\begin{proof}
One need to show that $P_{ik}$ and $Q_{ik}$ are orthogonal projections.
Obviuously, $P_{ik}$ and $Q_{ik}$ are selfadjoint. Moreover, 
$$
P_{ik}^2=U_{ik}^*U_{ik}U_{ik}^*U_{ik}=U_{ik}^*\left(\mathds{1}-\sum_{j\neq i}U_{jk}U_{jk}^*\right)U_{ik}=U_{ik}^*U_{ik}-\sum_{j\neq i}U_{ik}^*U_{jk}U_{jk}^*U_{ik}=U_{ik}^*U_{ik}=P_{ik}.
$$
The second equality follows from \eqref{e1.2a}, while the fourth from \eqref{e1.4a}. Thus, $P_{ik}$ is an orthogonal projection, hence $U_{ik}$ is a partial isometry. Similarly, using \eqref{e1.2b} and \eqref{e1.4b}, one can show that $Q_{ik}^2=Q_{ik}$ and this shows that $U_{ik}^*$ is a partial isometry.
\end{proof}

\begin{pro} 
The matrix $U=(U_{kl})\in M_n(B(H))$ is unitary, i.e. $U^*U=\mathds{1}_{M_n(B(H))}=UU^*$.
\end{pro}
\begin{proof} 
It follows from \eqref{e1.2b} that $U^*U=\mathds{1}_{M_n(B(H))}$. 
It remains to show that $UU^*=\jed_{M_n(B(H))}$. Firstly, notice that $UU^*$ is a projection. Moreover, it follows from \eqref{e1.4b} that
$(UU^*)_{kl}=\sum_{i}U_{ki}U_{li}^*=0$ 
for $k\neq l$. Therefore, diagonal entries of the matrix $UU^*$ are projections. Now, observe that 
$$ 
\sum_k(UU^*)_{kk}=\sum_{k}\left(\sum_{i} U_{ki}U_{ki}^*\right)=\sum_{i}\left(\sum_{k} U_{ki}U_{ki}^*\right)=\sum_{i}\mathds{1}=n\mathds{1},
$$ 
where the third equality follows from \eqref{e1.2a}. Since each diagonal term $(UU^*)_{kk}$ is a projection,  we conclude from Lemma \ref{er} that $(UU^*)_{kk}=\jed$ for each $k=1,\ldots,n$. Thus $UU^*=\jed_{M_n(B(H))}$.
\end{proof}
\begin{m} 
\label{uT}
By similar arguments one can show that the transpose matrix $U^T$ is also unitary.
\end{m}

\begin{cor}
\label{orto}
We have the following orthogonality relations:
\begin{equation}
\label{pik}
P_{ik}P_{il}=0\qquad \textrm{if $k\neq l$,}
\end{equation}
\begin{equation}
\label{qik}
Q_{ik}Q_{il}=0\qquad  \textrm{if $k\neq l$,}
\end{equation}
\begin{equation}
\label{pijk}
P_{ik}P_{jk}=0\qquad \textrm{if $i\neq j$,}
\end{equation}
\begin{equation}
Q_{ik}Q_{jk}=0\qquad \textrm{if $i\neq j$,}
\end{equation}
where $i,j,k,l=1,\ldots,n$.
\end{cor}
\begin{proof}
Let $i$ be fixed. Since the matrix $U^T$ is unitary,  
$\sum_k P_{ik}=\sum_k U_{ik}^*U_{ik}^{}=\jed$. As $P_{ik}$ are orthogonal projections, the equality in \eqref{pik} should be satisfied for every $k\neq l$ (cf. Lemma \ref{er}). The rest of the relations can be showed similarly using unitarity of $U$ and $U^T$. 
\end{proof}

Let $K_{ik}=\mathrm{Im}(P_{ik})=\mathrm{Im}(U_{ik}^*)$ and  $H_{ik}=\mathrm{Im}(Q_{ik})=\mathrm{Im}(U_{ik})$. 
\begin{cor} 
\label{zeros}
For any $i,k,l$ such that $k\neq l$,
\begin{equation}
U_{ik}^*U_{il}=0,
\eeq
\beq
\label{uiluikg}
U_{il}U_{ik}^*=0.
\end{equation}
\end{cor}
\begin{proof} Since $U_{ik}^{}U_{ik}^*$ and $U_{il}^{}U_{il}^*$ are orthogonal to each other then $\mathrm{Im}(U_{il})$ and $\mathrm{Im}(U_{ik})$ are orthogonal subspaces as well. Thus, $\mathrm{Im}(U_{il})\subset \mathrm{Im}(U_{ik})^\perp=\ker(U_{ik}^*)$, hence $U_{ik}^*U_{il}^{}=0$. By similar arguments one can show that also $U_{il}^{}U_{ik}^*=0$.
\end{proof}

\begin{pro}
\label{ret} 
For any $i,k,l$ such that $k\neq l$,
\begin{equation}
\label{uikuil}
U_{ik}U_{il}=0.
\end{equation}
\end{pro}
\begin{proof} 
For $i=j$ condition (\ref{e1.1}) reduces to 
\begin{equation}
\label{eqnow}
U_{ik}U_{il}=\omega_{kl}U_{il}U_{ik}.
\end{equation}
Observe, that $\ker(U_{ik}U_{il})\supset\ker(U_{il})=\mathrm{Im}(U_{il}^*)^\perp=K_{il}^\perp$. Similarly, $\ker(\omega_{kl}U_{il}U_{ik})\supset K_{ik}^\perp$. As the equality \eqref{eqnow} holds, $\ker(U_{ik}U_{il})\subset K_{il}^\perp+K_{ik}^\perp$. According to \eqref{pik} in Corollary \ref{orto} and Lemma \ref{lem2}, $K_{il}^\perp+K_{ik}^\perp=H$. Therefore, $U_{ik}U_{il}=0$.
\end{proof}

\begin{pro}
\label{norm}
Operators $U_{ik}$ are normal, i.e. $U_{ik}^{}U_{ik}^*=U_{ik}^*U_{ik}^{}$ for all $i,k=1,\ldots,n$.
\end{pro}
\begin{proof}
It follows from Proposition \ref{ret} that 
$U_{ik}^*U_{ik}^{}U_{il}^{}U_{il}^*=0$ 
for $i,k,l$ such that $k\neq l$.
Thus $H_{il}^{}\subset K_{ik}^\perp$ for $k\neq l$. Since 
$H=\bigoplus_{k} H_{ik}=\bigoplus_k K_{ik}$ for every $i$, we get $H_{ik}=K_{ik}$ for every $i,k$. The latter  is equivalent to $P_{ik}=Q_{ik}$. 
\end{proof}

\begin{cor} 
For $i,j,k$ such that $i\neq j$,
\begin{equation}
\label{uij}
U_{ik}U_{jk}=0.
\end{equation}
\end{cor}
\begin{proof} 
It follows from Corollary \ref{orto} and Proposition \ref{norm} that $P_{ik}Q_{jk}=P_{ik}P_{jk}=0$. Thus $\mathrm{Im}(U_{jk})\subset\mathrm{Im}(U_{ik}^*)^\perp=\ker(U_{ik})$, and \eqref{uij} follows.
\end{proof}

\begin{pro}
\label{pro1.11} 
For $i,j,k,l$ such that $i\neq j$ and $k\neq l$,
\begin{equation}
\label{e1.51}
U_{ik}U_{jl}=\omega_{ji}\omega_{kl}U_{jl}U_{ik}.
\end{equation}
\end{pro}
\begin{proof} 
Let $i\neq j$ and $k\neq l$. Condition \eqref{e1.1} is equivalent to
$$
U_{ik}U_{jl}-\omega_{ji}\omega_{kl}U_{jl}U_{ik}= \omega_{kl}U_{il}U_{jk}-\omega_{ji}U_{jk}U_{il}. 
$$
Let $L$ and $R$ denote respectively the left and the right hand side of the above equality. Observe that 
$$
\ker (L)\supset\ker(U_{ik})\cap\ker(U_{jl})=\mathrm{Im}(U_{ik}^*)^\perp\cap\mathrm{Im}(U_{jl}^*)^\perp=K_{ik}^\perp\cap K_{jl}^\perp=
(K_{ik}+K_{jl})^\perp.
$$
Similarly, $\ker (R)\supset (K_{il}+K_{jk})^\perp$. Notice that due to Corollary \ref{orto} subspaces $K_{ik}+K_{jl}$ and $K_{il}+K_{jk}$ are orthogonal to each other. Since $L=R$, we arrive at
$$
\ker(L)\supset  (K_{ik}+K_{jl})^\perp+(K_{il}+K_{jk})^\perp= H. $$
The last equality is due to Lemma \ref{lem2}. Therefore, $L=0$.
\end{proof}

\begin{pro} 
\label{uug}
For $i,j,k,l$ such that $i\neq j$ and $k\neq l$,
\begin{equation}
\label{glsp}
U_{ik}^*U_{jl}=\omega_{ij}\omega_{lk}U_{jl}U_{ik}^*.
\end{equation}
\end{pro}
\begin{proof} Since $U_{ik}U_{ik}^*=U_{ik}^*U_{ik}$ (cf Proposition \ref{norm}), operators $U_{ik},U_{ik}^*$ restricted to $H_{ik}$ are isomorphisms and they are zero on the orthogonal complement of the subspace $H_{ik}$ for every $i,k$. Now, let $i\neq j$ and $k\neq l$. Since $U_{jl}(H_{jl})=H_{jl}$, due to Proposition \ref{pro1.11} we have
$U_{ik}(H_{jl})=U_{ik}U_{jl}(H_{jl})=U_{jl}U_{ik}(H_{jl})\subset H_{jl}$. Therefore, $U_{ik}(H_{ik}\cap H_{jl})\subset H_{ik}\cap H_{jl}$ and $U_{jl}(H_{ik}\cap H_{jl})\subset H_{ik}\cap H_{jl}$.  
Similarly, one can show that $U_{ik}^*(H_{jl})\subset H_{jl}$, and $U_{ik}^*(H_{ik}\cap H_{jl})\subset H_{ik}\cap H_{jl}$ and $U_{jl}^*(H_{ik}\cap H_{jl})\subset H_{ik}\cap H_{jl}$.
If $x\in H_{ik}\cap H_{jl}$, then due to \eqref{e1.51} we have 
$$
U_{ik}U_{jl}U_{ik}^*(x)=\omega_{ji}\omega_{kl}U_{jl}U_{ik}U_{ik}^*(x)=\omega_{ji}\omega_{kl}U_{jl}(x)
$$
and consequently
\begin{equation}
\label{ew}
U_{jl}U_{ik}^*(x)=U_{ik}^*U_{ik}U_{jl}U_{ik}^*(x)=\omega_{ji}\omega_{kl}U_{ik}^*U_{jl}(x), \qquad x\in H_{ik}\cap H_{jl}.
\end{equation}
On the other hand, if $x\in(H_{ik}\cap H_{jl})^\perp$, then by Lemma \ref{lem2}, 
$x=\sum_{r\neq i}x_{rk} +\sum_{s\neq j}x_{sl}$ where $x_{rk}\in H_{rk}$ and $x_{sl}\in H_{sl}$. Thus 
$$ 
U_{jl}U_{ik}^*(x)=\sum_{r\neq i}U_{jl}U_{ik}^*(x_{rk})+\sum_{s\neq j}U_{jl}U_{ik}^*(x_{sl}).  
$$ 
Observe that $U_{ik}^*(x_{rk})=0$ for $r\neq i$ (cf Corollary \ref{orto} and Proposition \ref{norm}). Further, if $s\neq j$, then $U_{ik}^*(x_{sl})\in H_{sl}$ for $s\neq i$ and $U_{ik}^*(x_{sl})=0$ for $s=i$. Hence $U_{jl}U_{ik}^*(x_{sl})=0$ (again Corollary \ref{orto}). Thus we came to equality $U_{jl}U_{ik}^*(x)=0$ for $x\in (H_{ik}\cap H_{jl})^\perp$. In a similar way we show that $U_{ik}^*U_{jl}(x)=0$ for $x\in  (H_{ik}\cap H_{jl})^\perp$, so we get 
\beq
\label{eo}
U_{jl}U_{ik}^*(x)=0=\om_{ji}\om_{kl}U_{ik}^*U_{jl}(x), \qquad x\in(H_{ik}\cap H_{jl})^\perp.
\eeq
Taking into account \eqref{ew} and \eqref{eo} we get \eqref{glsp}.
\end{proof} 

\begin{cor}
\label{kom}
For every $i,j,k,l$,
\be
[U_{ik},P_{jl}] & = & 0, \label{komup} \\
{[P_{ik},P_{jl}]} & = & 0. \label{kompp}
\ee
\end{cor}
\begin{proof}
If $i=j$ or $k=l$, then from \eqref{uiluikg}, \eqref{uikuil}  or \eqref{e1.4b}, \eqref{uij} respectively we get $U_{ik}P_{jl}=0=P_{jl}U_{ik}$. Otherwise,
$$U_{ik}P_{jl}=U_{ik}U_{jl}^*U_{jl}=\om_{ij}\om_{lk}U_{jl}^*U_{ik}U_{jl}= \om_{ij}\om_{lk}\om_{ji}\om_{kl}U_{jl}^*U_{jl}U_{ik}=P_{jl}U_{ik},$$
so \eqref{komup} is satisfied. Relation \eqref{kompp} is a consequence of \eqref{komup}.
\end{proof}

For every permutation $\sigma\in S_n$ we define a projection $P_\s$ by
\begin{equation}
\label{psigma}
P_{\sigma}=\prod_{i=1}^n P_{i,\sigma^{-1}(i)}=P_{1,\s^{-1}(1)}P_{2,\s^{-1}(2)}\ldots P_{n,\sigma^{-1}(n)}.
\end{equation}
The projections satisfy $P_\sigma P_{\tau}=\delta_{\sigma,\tau}P_\s$ for $\sigma, \tau\in S_n$, because if $\sigma\neq \tau$ then there is 
$i$ such that $\si(i)\neq\ti(i)$ and  $P_{i,\si(i)}P_{i,\ti(i)}=0$. Moreover, we have
\begin{pro}
\label{cent}
The projections $P_\s$, $\s\in S_n$, are central and they form a resolution of the identity, i.e.
$ 
\sum_{\sigma\in S_n} P_{\sigma}=\mathds{1}.
$ 
\end{pro}
\begin{proof}
It follows from unitarity of the matrix $U^T$ (cf. Remark \ref{uT}) that $\sum_{k}P_{ik}=\jed$ for every $i$. Thus,
\beg 
\jed & = &
\sum_{k_1,k_2,\ldots,k_n}P_{1,k_1}P_{2,k_2}\ldots P_{n,k_n} =
\sum_{\substack{k_1,k_2,\ldots,k_n \\ \textrm{$k_i\neq k_j$ for $i\neq j$}}}P_{1,k_1}P_{2,k_2}\ldots P_{n,k_n} \\
&=&
\sum_{\s\in S_n} P_{1,\s^{-1}(1)}P_{2,\s^{-1}(2)}\ldots P_{n,\sigma^{-1}(n)}
=\sum_{\s\in S_n}P_\s
\eeg
The second equality follows from orthogonality relations \eqref{pik}.
Projections $P_\s$ are central due to Corollary \ref{kom}.
\end{proof}



\begin{proof}[Proof of Theorem \ref{iso}]
Let $v_{ik}\in A_\te^n$, $i,k=1,\ldots,n$, be elements defined by
\beq
\label{vik}
v_{ik}=\sum_{\substack{\s\in S_n \\ \s(k)=i}} x_{\s,i}.
\eeq
We will show that elements $v_{ik}$ satisfy relations of Proposition \ref{profil}.
Observe that due to \eqref{mult1}
$$
v_{ik}v_{jl}=\sum_{\substack{\s\in S_n \\ \s(k)=i}}\;\sum_{\substack{\t\in S_n \\ \t(l)=j}}x_{\s,i} x_{\t,j}=\sum_{\substack{\s\in S_n \\ \s(k)=i,\,\s(l)=j}}x_{\s,i}x_{\s,j}.
$$
If $i=j$, $k\neq l$ or $i\neq j$, $k=l$, then $v_{ik}v_{jl}=0$, similarly $v_{il}v_{jk}=0$, $v_{jk}v_{il}=0$ and $v_{jl}v_{ik}=0$. Thus, the relation \eqref{e1.1} is satisfied.
Assume $i\neq j$ and $k\neq l$. Then
\beg
v_{ik}v_{jl}&=&
\sum_{\substack{\s\in S_n \\ \s(k)=i,\,\s(l)=j}}x_{\s,i}x_{\s,j}
\\&=&
\sum_{\substack{\s\in S_n \\ \s(k)=i,\,\s(l)=j}}\om_{ij}^{(\s)} x_{\s,j}x_{\s,i}
\\&=&
\sum_{\substack{\s\in S_n \\ \s(k)=i,\,\s(l)=j}}\om_{ji}\om_{\si(i),\si(j)}\, x_{\s,j}x_{\s,i}
\\&=&
\om_{ji}\om_{kl}\sum_{\substack{\s\in S_n \\ \s(k)=i,\,\s(l)=j}} x_{\s,j}x_{\s,i}
\\&=&
\om_{ji}\om_{kl}\,v_{jl}v_{ik}.
\eeg
Analogously,
$v_{jk}v_{il}=\om_{ij}\om_{kl}\,v_{il}v_{jk}$, so relation \eqref{e1.1} follows.
Further, notice that
\beq
\label{ts}
x_{\s,i}^{} x_{\s,i}^*=x_{\s,j}^{} x_{\s,j}^*
\eeq
for every $\s\in S_n$ and $i,j=1,\ldots,n$. Indeed, since $x_{\s,i}^* x_{\t,j}^{}=0$ for $\s\neq\t$ (cf. \eqref{mult2}), we infer from \eqref{mult3}
$$
x_{\s,i}^{} x_{\s,i}^*=x_{\s,i}^{} x_{\s,i}^*\sum_\t x_{\t,j}^{} x_{\t,j}^*= x_{\s,i}^{} x_{\s,i}^*x_{\s,j}^{} x_{\s,j}^*= 
\sum_\t x_{\t,i}^{} x_{\t,i}^* x_{\s,j}^{} x_{\s,j}^*= x_{\s,j}^{} x_{\s,j}^*.
$$
For any $k,l$ we have from \eqref{ts} and \eqref{mult3}
\beg
\sum_i v_{ik}^{} v_{il}^*&=& \sum_i \sum_{\substack{\s\in S_n \\ \s(k)=i}}\; \sum_{\substack{\t\in S_n \\ \t(l)=i}}x_{\s,i}^{} x_{\t,i}^*
\\& = &
\delta_{k,l}\sum_i\sum_{\substack{\s\in S_n \\ \s(k)=i}}x_{\s,i}^{} x_{\s,i}^* 
\\&=& 
\delta_{k,l}\sum_i\sum_{\substack{\s\in S_n \\ \s(k)=i}}x_{\s,1}^{} x_{\s,1}^* 
\\&=&
\delta_{k,l}\sum_{\s\in S_n}x_{\s,1}^{} x_{\s,1}^*=\delta_{k,l}\jed.
\eeg
Thus, relation \eqref{e1.2a} is satisfied. Similarly, one can prove \eqref{e1.2b}.
Finaly, if $i\neq j$, then
$$
v_{ik}^{} v_{jk}^*=\sum_{\substack{\s\in S_n \\ \s(k)=i}}\; \sum_{\substack{\t\in S_n \\ \t(k)=j}}x_{\s,i}^{} x_{\t,j}^*=0
$$
because $x_{\s,i}^{} x_{\t,j}^*=0$ for $\s\neq\t$ and there is no $\s$ such that $\s(k)=i$ and $\s(k)=j$. Therefore, \eqref{e1.4a} holds. Similar arguments for \eqref{e1.4b}.

Next, we will show, that the algebra $A_\te^n$ is generated by elements $v_{ik}$. 
Notice that for every $\s$ and $i$,
$$
x_{\s,i}x_{\s,i}^*x_{\s,i}^{}=  x_{\s,i}^{}\left(\jed-\sum_{\t\neq\s}x_{\t,i}^*x_{\t,i}^{}\right)= x_{\s,i}^{}-\sum_{\t\neq\s}x_{\s,i}^{}x_{\t,i}^*x_{\t,i}^{} = x_{\s,i}^{}.
$$
Consequently $(x_{\s,i}^*x_{\s,i}^{})^2=x_{\s,i}^*x_{\s,i}^{}$, so $x_{\s,i}^*x_{\s,i}^{}$ is a projection. 
Observe also that
$$
v_{i,\si(i)}^*v_{i,\si(i)}^{}=\sum_{\substack{\t\in S_n \\ \ti(i)=\si(i)}}x_{\t,i}^*x_{\t,i}^{}
$$
and
$$
(v_{1,\si(1)}^*v_{1,\si(1)}^{})(v_{2,\si(2)}^*v_{2,\si(2)}^{})\ldots (v_{n,\si(n)}^*v_{n,\si(n)}^{})=
 (x_{\s,1}^*x_{\s,1}^{})(x_{\s,2}^*x_{\s,2}^{})\ldots (x_{\s,n}^*x_{\s,n}^{})
$$
Hence,
\be
x_{\s,i}&=&
x_{\s,i}^{} x_{\s,i}^*x_{\s,i}^{} \label{xsi}
\\&=&
v_{i,\si(i)} x_{\s,i}^*x_{\s,i}^{} \nonumber
\\&=&
v_{i,\si(i)} (x_{\s,i}^*x_{\s,i}^{})^n \nonumber
\\&=&
v_{i,\si(i)} (x_{\s,1}^*x_{\s,1}^{})(x_{\s,2}^*x_{\s,2}^{})\ldots (x_{\s,n}^*x_{\s,n}^{}) \nonumber
\\&=&
v_{i,\si(i)}(v_{1,\si(1)}^*v_{1,\si(1)}^{})(v_{2,\si(2)}^*v_{2,\si(2)}^{})\ldots (v_{n,\si(n)}^*v_{n,\si(n)}^{}) \nonumber
\ee
Since elements $x_{\s,i}$ generate $A_\te^n$, elements $v_{ik}$ are generators of $A_\te^n$ too.

Finally, let $U_{ik}\in B(H)$ for some Hilbert space $H$, and assume that they satisfy relations \eqref{e1.1} -- \eqref{e1.4b}. We will show that there is a unique representation $\pi:A_\te^n\to B(H)$ such that $\pi(v_{ik})=U_{ik}$ for $i,k=1,\ldots,n$. For $\s\in S_n$ and $i=1,\ldots,n$, define operators $X_{\s,i}\in B(H)$ by 
\beq
\label{Xsi}
X_{\s,i}=U_{i,\si(i)}P_\s,
\eeq
where $P_\s$ is the projection defined in \eqref{psigma}.
It follows from Proposition \ref{cent}, that $X_{\s,i}^{}X_{\t,j}^{}=0=X_{\s,i}^{}X_{\t,j}^*$ for $\s\neq \t$. Proposition \ref{pro1.11} impies
$$
X_{\s,i}X_{\s,j}=U_{i,\si(i)}U_{j,\si(j)}P_\s=\om_{j,i}\om_{\si(i),\si(j)}U_{j,\si(j)}U_{i,\si(i)}P_\s=\om_{ij}^{(\s)}X_{\s,j}X_{\s,i},
$$
while Proposition \ref{uug} leads to
$$
X_{\s,i}^{}X_{\s,j}^*=U_{i,\si(i)}^{}U_{j,\si(j)}^*P_\s=\om_{i,j}\om_{\si(j),\si(i)}U_{j,\si(j)}^*U_{i,\si(i)}^{}P_\s=\om_{ji}^{(\s)}X_{\s,j}^*X_{\s,i}^{}
$$
for $i\neq j$. Since $U_{ik}$ are normal (cf. Proposition \ref{norm}), we have also $X_{\s,i}^{}X_{\s,i}^*=P_\s=X_{\s,i}^*X_{\s,i}^{}$. Proposition \ref{cent} leads to 
$$
\sum_{\s}X_{\s,i}^{} X_{\s,i}^*=\sum_\s X_{\s,i}^*X_{\s,i}^{} =\sum_{\s}P_\s=\jed.
$$
Thus, all relations \eqref{mult1} -- \eqref{mult3} are satisfied by operators $X_{\s,i}$. Since $A_\te^n$ is a universal algebra subject to these relations, there exists a unique representation $\pi:A_\te^n \to B(H)$ such that $\pi(x_{\s,i})=X_{\s,i}$ for every $\s$ and $i$. Now, observe that $U_{ik}P_\s=0$ if and only if $\s(k)\neq i$. Hence
$$U_{ik}=U_{ik}\sum_{\s\in S_n}P_\s=U_{ik}\sum_{\substack{\s\in S_n \\ \s(k)=i}}P_\s=\sum_{\substack{\s\in S_n \\ \s(k)=i}} X_{\s,i}.$$
Comparing it with \eqref{vik} we get $\pi(v_{ik})=U_{ik}$. On the other hand, if $\pi':A_\te^n\to B(H)$ is any representation such that $\pi'(v_{ik})=U_{ik}$, then $\pi'(x_{\s,i})=X_{\s,i}$ by \eqref{xsi} and \eqref{Xsi}, and consequently $\pi'=\pi$.
\end{proof}
\begin{pro}
The C*-algebra $A_\te^n$ is nuclear.
\end{pro}
\begin{proof}
For $\s\in S_n$, let $\gamma^{(\s)}$ be the action of $\bT^n$ on $C(\bT_{\te^{(\s)}}^n)$ given by \eqref{gama}. Consider an action $\tilde{\gamma}=\bigoplus_\s\gamma^{(\s)}$ of $\bT^n$ on $A_\te^n$, i.e.
$ 
\tilde{\gamma}_t(x_{\s,i})=t_ix_{\s,i}
$ 
for $\s\in S_n$, $i=1,\ldots,n$. The fixed point C*-subalgebra $(A_\te^n)^{\tilde{\gamma}}$ is the linear span of central projections $p_\s$ (cf. Proposition \ref{cent}), i.e. $(A_\te^n)^{\tilde{\gamma}}$ is isomorphic to $C(S_n)$. Since $(A_\te^n)^{\tilde{\gamma}}$ is commutative, it is a nuclear C*-algebra. Hence, according to \cite[Proposition 2]{DLRZ} $A_\te^n$ is nuclear too.
\end{proof}
\begin{m}
Let us consider the case when $\omega_{ij}$ are trivial, i.e. $\omega_{ij}=1$ for all $i,j$ (or equivalently $\theta=0$). Then $A^n_0$ is a commutative C*-algebra which is isomorphic to the algebra $C(\mathds{T}^n\rtimes S_n)$ of continuous functions on semidirect product $\mathds{T}^n\rtimes S_n$. To verify this one can think of $\mathds{T}^n\rtimes S_n$ as a subgroup in a group of $n\times n$ unitary matrices $U(n)$, i.e. there is a faithful representation $\rho:\mathds{T}^n\rtimes S_n\rightarrow U(n)$ such that
\begin{equation}
\rho((t,\sigma))_{i,j}=t_i\delta_{i,\sigma(j)}
\end{equation}where $(t,\sigma)\in \mathds{T}^n\rtimes S_n$. Therefore, algebra of continuous functions on $\mathds{T}^n\rtimes S_n$ can be seen as a universal C*-algebra generated by matrix entries functions $\rho_{i,j}:\mathds{T}^n\rtimes S_n\rightarrow \mathbb{C}$. It is easy to check that such generators are subjected to relations \eqref{e1.1} -- \eqref{e1.4b} with $\omega_{ij}=1$ for all $i,j$.
\end{m}
\begin{m}
\label{uwHajac}
It is worth noting that for $n=2$ the C*-algebra $A_\theta^2$ is in fact the algebra of the quantum double torus $D_q\bT^2$ (described by Hajac and Masuda \cite{HM98}) if one puts $\omega_{21}=q$. It is not surprising since the structure of $D_q\bT$ was obtain as a deformation of the algebra of continuous functions on the semidirect product $\mathds{T}^2\rtimes \mathds{Z}_2$, where $\mathds{Z}_2$ can be treated as $S_2$.
\end{m}

\section{Compact quantum group $\bG_\te$
}
\subsection{General properties}
The disscusion in previous section provides necessary tools for proving the following two central propositions. In the sequel, due to Theorem \ref{iso}, we can identify generators $v_{ik}$ of $A_\te^n$ with generators $u_{ik}$ of the universal algebra $A_u$.
\begin{pro} 
\label{cqgte}
The pair $\bG_\te=(A_\theta^n,\Delta)$ with comultiplication $\Delta:A_\te^n\to A_\te^n\otimes A_\te^n$ given by
\begin{equation}
\label{deluik}
\Delta(u_{ik})=\sum_j u_{ij}\otimes u_{jk}
\end{equation}
is a compact quantum group.
\end{pro}
\begin{proof}
Straightforward computation shows that $\Delta$ is compatible with all relations \eqref{e1.1} -- \eqref{e1.4b}, so it can be uniquely extended to the unital $^*$-homomorphism on the whole $A_\theta^n$. To show that $\Delta$ satisfies condition (1) of Definition \ref{cqg} it is enough to cheeck that equality for generators $u_{ik}$, what seems to be a rather easy exercise. Finally, let us recall that $A_\theta^n$ is a universal C*-algebra generated by entries of the invertible matrix, so the condition (2) of Definition \ref{cqg} is also satisfied (for details see \cite[Proposition 3.6]{MD98}).
\end{proof}

Observe that the dense *-subalgebra $\cA_\te^n$ generated by matrix coefficients $u_{ik}$ can be described as follows. Remind that there are central projections $p_\s$, $\s\in S_n$, such that $\sum_\s p_\s=\jed$, where
$$p_\s=\prod_{i=1}^n u_{i,\si(i)}^{}u_{i,\si(i)}^*$$
 (cf. Proposition \ref{cent}). Let us introduce the following notations. For $\s\in S_n$ and 
$i=1,\ldots,n$ let
$$x_{\s,i}=u_{i,\si(i)}p_\s,$$
and for $\s\in S_n$, $\vr=(r_1,\ldots,r_n)\in\bZ^n$, let
\beq
u_\s^\vr=u_{1,\si(1)}^{r_1}u_{2,\si(2)}^{r_2}\ldots u_{n,\si(n)}^{r_n},
\eeq
where $u_{ik}^r=(u_{ik}^*)^{-r}$ for $r<0$ and $u_{i\sigma^{-1}(i)}^0=p_\sigma$ for any $\sigma$ .
Since $u_{ik}^\#u_{jl}^\#=0$ whenever $i=j$, $k\neq l$ or $i\neq j$, $k=l$ (where $a^\#$ is $a$ or $a^*$), we have
$$
\cA_\te^n=\mathrm{span}\{u_\s^\vr:\,\s\in S_n,\,\vr\in\bZ^n\}.
$$
Remind that for fixed $\s$, elements $x_{\s,i}$, $i=1,\ldots,n$, generate the algebra $C(\bT_{\te^{(\s)}}^n)$, and 
$$\mathrm{Poly}(C(\bT_{\te^{(\s)}}^n))=\mathrm{span}\{x_\s^\vr:\,\vr\in\bZ^n\},$$
where $x_\s^\vr=x_{\s,1}^{r_1}\ldots x_{\s,n}^{r_n}$. Therefore
$$\cA_\te^n=\bigoplus_{\s\in S_n}\mathrm{Poly}(C(\bT_{\te^{(\s)}}^n)).$$
One can check that the counit $\varepsilon:\cA_\te^n\to \bC$ and coinverse $\kappa:\cA_\te^n\to \cA_\te^n$ are defined on generators as
\beq
\varepsilon(x_{\s,i})=\delta_{\s,e},
\eeq
\beq
\kappa(x_{\s,i})=x_{\si,\si(i)}^*,
\eeq
where $e$ denotes the unit of the group $S_n$.


Now, we define a coaction of $\bG_\te$ on the algebra $C(\bT_\te^n)$ of the noncommutative $n$-torus. Let us remind (cf. Example \ref{exfol}), that we consider subspaces $V_{p,q}\subset C(\bT_\te^n)$, $p,q=0,1,2,\ldots$, where
$$
V_{p,q}=\mathrm{span}\{x_{i_1}^{} \ldots x_{i_p}^{} x_{j_q}^*\ldots x_{j_1}^*:\,\mbox{$i_s\neq j_t$ for every $s=1,\ldots,p$ and $t=1,\ldots,q$}\}
$$
The family of these subspaces together with the faithful trace $\phi$ on $C(\bT_\te^n)$ forms an orthogonal filtration $\cV$ on $C(\bT_\te^n)$. 
\begin{thm} 
There exists a coaction $\alpha:C(\mathds{T}_\theta^n)\rightarrow C(\mathds{T}_\theta^n)\otimes A_\theta^n$ of $\bG_\te$ on $C(\bT_\te^n)$ defined by 
\begin{equation}
\label{alfa}
\alpha(x_k)=\sum_i x_i\otimes u_{ik},
\end{equation}
where $\left\{x_i:i=1,2\ldots,n\right\}$ is a set of generators of $C(\mathds{T}_\theta^n)$, $\left\{u_{ik}:i,k=1,2,\ldots,n\right\}$ is a set of generators of $A_\theta^n$ and $\alpha$ acts in a filtration preserving way.

Moreover, the quantum group $\bG_\te$ is the quantum symmetry group of the noncommutative $n$-torus $\bT_\te^n$.
\end{thm}
\begin{proof}
Since generators $u_{ik}$ satisfy the relations \eqref{e1.1} -- \eqref{e1.4b}, condition \eqref{alfa} determines a uniqe unital $^*$-homomorphism $\alpha$ on $C(\mathds{T}_\theta^n)$. To show that $\alpha$ satisfies condition (1) of Definition \ref{act} it is enough to verify it only for generators $u_{ik}$. It follows from straightforward calculations. 
Further, for any $k$, unitarity of the matrix $u$ implies
$$ 
\sum_i\alpha \left(x_i\right) \left(\mathds{1}_{C(\mathds{T}_\theta^n)}\otimes u^*_{ki}\right)=\sum_{i,j}x_j\otimes u_{ji}u_{ki}^*=
\sum_{j}x_j\otimes \delta_{jk}\mathds{1}_{A^n_{\theta}}=x_k\otimes \mathds{1}_{A^n_{\theta}}.
$$ 
Hence condition (2) of Definition \ref{act} follows. To end the proof we should show that $\alpha(V_{p,q})\subset V_{p,q}\otimes A_\theta^n$. 
Indeed, if $k_s\neq l_t$ for $s=1,\ldots,p$ and $t=1,\ldots,q$, then
\beg
\alpha(x_{k_1}^{} \ldots x_{k_p}^{} x_{l_q}^*\ldots x_{l_1}^*) &=&
\sum_{i_1,\ldots,i_p,j_1,\ldots,j_q}x_{i_1}^{} \ldots x_{i_p}^{} x_{j_q}^*\ldots x_{j_1}^* \otimes u_{i_1,k_1}^{}\ldots u_{i_p,k_p}^{} u_{j_q,l_q}^*\ldots u_{j_1,l_1}^*
\eeg
If the product $u_{i_1,k_1}^{}\ldots u_{i_p,k_p}^{} u_{j_q,l_q}^*\ldots u_{j_1,l_1}^*$ is nonzero, then it follows from Corollary \eqref{zeros} that $i_s\neq j_t$ for every $t,s$. Thus $\alpha(u_{i_1,k_1}^{}\ldots u_{i_p,k_p}^{} u_{j_q,l_q}^*\ldots u_{j_1,l_1}^*)\in V_{p,q}\otimes A_\te^n$.

It remains to show that $(\bG_\te,\alpha_\te)$ is a final object in the category $\cC_{C(\bT_\te^n),\cV}$ of compact quantum groups acting in the filtration preserving way on $C(\bT_\te^n)$ (cf. Subsection \ref{subs_qsg} and \cite{BS13}). Assume that $(\bG, \alpha_\bG)$ is an object in the category $\cC_{C(\bT_\te^n),\cV}$, where $\bG=(A_\bG,\Delta_\bG)$. Since $\alpha_\bG(V_{1,0})\subset V_{1,0}\otimes A_\bG$, there are uniquely determined elements $w_{ik}\in A_\bG$ such that 
$\alpha_\bG(x_k)=\sum_{i=1}^n x_i\otimes w_{ik}.$
Clearly, the elements form a representation. Consequently, there is a unital *-homomorphism $\pi: A_\te^n\to A_\bG$ such that $\pi(u_{ik})=w_{ik}$. One easily checks that both conditions \eqref{univ1} and \eqref{univ2} are satisfied. Moreover, if $\pi': A_\te^n\to A_\bG$ is another *-homomorphism satisfying these conditions, then 
$$\sum_i x_i\otimes \pi'(u_{ik})=\sum_i x_i\otimes w_{ik}.$$
Thus, $\pi'(u_{ik})=w_{ik}=\pi(u_{ik})$, hence $\pi'=\pi$. 
\end{proof}

\subsection{Haar measure}
Before we go further, let us consider some useful structural results. 
For each $\vr\in\bZ^n$ consider a function $\tilde{\psi}_\vr:S_n\times S_n\to \bT$ given by
$$ 
\tilde{\psi}_{\vec{r}}(\t,\s)=\prod_{\substack {i<j\\ \tau^{-1}(i)>\tau^{-1}(j)}} \left(\om_{\t^{-1}(j),\t^{-1}(i)}\om_{\s^{-1}(i),\s^{-1}(j)}\right)^{r_ir_j}
$$ 
Making substitution $\t^{-1}(i)\mapsto j$, $\t^{-1}(j)\mapsto i$ one can observe that
\beq
\label{psi-tilde}
\tilde{\psi}_{\vec{r}}(\t,\s)=\prod_{\substack {i<j\\ \tau(i)>\tau(j)}} \left(\om_{ij}\om_{\s^{-1}\t(j),\s^{-1}\t(i)}\right)^{r_{\t(i)}r_{\t(j)}}
\eeq
\color{black}
For $\vec{r}
\in\bZ^n$ and $\s\in S_n$ let $\s\vec{r}=(r_{\s^{-1}(1)},\ldots,r_{\s^{-1}(n)})$. Define 
$$u_{\t,\s}^{\vec{r}}=u_{\ti(1),\s^{-1}(1)}^{r_1}u_{\ti(2),\s^{-1}(2)}^{r_2}\ldots u_{\ti(n),\s^{-1}(n)}^{r_n}
$$
for $\s,\t\in S_n$.
One can easily check that 
\beq
\label{porz}
u_{\t,\s}^{\vec{r}}=\tilde{\psi}_{\vec{r}}(\t,\s)u_{\ti\s}^{\ti\vec{r}}.
\eeq

Further, by \eqref{deluik} we get 
$$ 
\Delta(p_{ik})=\sum_lp_{il}\otimes p_{lk}.
$$ 
Due to orthogonality relations \eqref{pijk} it leads to
\begin{eqnarray}
\Delta(p_\sigma)
&=&
\sum_{l_1,\ldots l_n}p_{1,l_1}\ldots p_{n,l_n}\otimes p_{l_1,\sigma^{-1}(1)}\ldots p_{l_n,\sigma^{-1}(n)}\label{deltap}
\\
&=&\sum_{\tau}p_{1,\ti(1)}\ldots p_{n,\ti(n)}\otimes p_{\ti(1),\sigma^{-1}(1)}\ldots p_{\ti(n),\sigma^{-1}(n)}
\nonumber\\
&=&\sum_{\tau}p_{1,\ti(1)}\ldots p_{n,\ti(n)}\otimes p_{1,\sigma^{-1}\t(1)}\ldots p_{n,\sigma^{-1}\t(n)}\nonumber\\
&=&\sum_{\t} p_{\t}\otimes p_{\ti\sigma}.\nonumber
\end{eqnarray}
Relations \eqref{deluik} and \eqref{porz} imply 
\begin{eqnarray}
\Delta(u_\sigma^\vr)&=&
\sum_{l_1,l_2\ldots l_{n}}u^{r_1}_{1,l_1}\ldots u^{r_n}_{n,l_{n}}\otimes u^{r_1}_{l_1,\sigma^{-1}(1)}\ldots u^{r_n}_{l_{n},\sigma^{-1}(n)} \label{deltau}\\
&=&\sum_{\t}u^{r_1}_{1,\ti(1)}
\ldots u^{r_n}_{n,\ti(n)}\otimes u^{r_1}_{\ti(1),\sigma^{-1}(1)}
\ldots u^{r_n}_{\ti(n),\sigma^{-1}(n)}\nonumber\\
&=&
\sum_{\t}u_\t^\vr\otimes u_{\t,\s}^\vr
=\sum_{\tau}\tilde{\psi}_\vr(\tau,\sigma) u_{\t}^\vr\otimes u_{\ti\sigma}^{\ti\vr},
\nonumber
\end{eqnarray}
By (\ref{deltap}) the above formula is true also for $\vr$ with at least one $r_i=0$.
Consequently 
\beq
\label{deltax}
\Delta(x_{\s}^\vr)=\sum_\t \tilde{\psi}_\vr(\tau,\sigma)x_\t^\vr \otimes x_{\ti\s}^{\ti\vr}.
\eeq

Now, let us consider the Haar measure $h$ on $\bG_\te$. Since the underlying algebra decomposes as a direct sum $A^n_\theta=\bigoplus_{\sigma}C(\bT_{\te^{(\s)}}^n)$ it is enough to specify $h$ on the elements which lineary generate each component $C(\bT_{\te^{(\s)}}^n)$.
\begin{pro} 
Haar measure $h$ on $\bG_\te
$ satisfies 
\beq
\label{haar}
h(x_\s^\vr)=\frac{1}{n!}\delta_{\vec{0},\vr}.
\eeq
\end{pro}
\begin{proof} Since Haar measure is unique, it is enough to check that functional defined by \eqref{haar} is a Haar measure. 
$$ 
h(\mathds{1})=
\sum_\sigma h(p_\sigma)=
\sum_\s h(x_\s^{\vec{0}})=1.
$$ 
From \eqref{deltax} 
we get 
$$ 
(h\otimes \id)\Delta(x_\sigma^\vr)=\sum_{\tau}\tilde{\psi}_\vr(\tau,\sigma) h(x_\tau^\vr) x^{\tau^{-1}\vr}_{\ti\s}=\delta_{\vr,\vec{0}}\frac{1}{n!}\sum_{\tau} x_{\tau^{-1}\sigma}^{\vec{0}}=\delta_{\vr,\vec{0}}\frac{1}{n!}\sum_\t p_{\ti\s}=\delta_{\vr,\vec{0}}\frac{1}{n!}\mathds{1}=
h(x_\s^\vr)\jed
$$ 
and similarly 
$(\id\otimes h)\Delta(x_\sigma^\vr)=h(x_\s^\vr)\jed.$
\end{proof}
\begin{m}
It follows from above proposition that $h=\bigoplus_\s \phi_\s$, where $\phi_\s$ is a faithful trace on $C(\bT_{\te^{(\s)}}^n)$. Therefore, $h$ is a faithful trace on $A_\te^n$. It means particularly that $\bG_\te$ is of Kac type. 
\end{m}

\subsection{Representations of $\bG_\te$}
If $\s\in S_n$, then we denote $A_{\te,\s}^n=C(\bT_{\te^{(\s)}}^n)$.
For a subgroup $H\subset S_n$ define $A_H=\bigoplus_{\s\in H}A_{\theta,\s}^n$. Clearly, $A_H$ is a subalgebra of $A_{\theta}^n$. For $i,j=1,\ldots,n$, we write $i\sim_H j$ if $i=\sigma(j)$ for some $\s\in H$. Clearly, $\sim_H$ is an equivalence relation. One has
\beq
\label{genAH}
A_H=\mathrm{C}^*\{u_{ij}: \,\text{$i,j=1,\ldots,n$ and $i\sim_Hj$}\} 
\eeq
Let $p_H=\sum_{\sigma\in H}p_\s$. Then $p_H$ is a maximal central projection in $A_H$, i.e. $A_H=p_HA_\te^n$. Moreover, for every $i,j$ such that $i\sim_H j$
\beq
\label{unitarH}
\sum_{k\sim_H i}u_{ik}^{}u_{jk}^*=\delta_{i,j}p_H=\sum_{k\sim_H i}u_{ki}^*u_{kj}^{}
\eeq
Define $\Delta_H:A_H\to A_H\otimes A_H$ on generators described in \eqref{genAH} by 
$$\Delta_H(u_{ij})=\sum_{k\sim_H j}
u_{ik}\otimes u_{kj}$$
It follows from transitivity of the relation $\sim_H$ that $\Delta_H$ is well defined. \begin{pro}
\label{AH}
The pair $(A_H,\Delta_H)$ is a compact quantum group. 

Moreover, a morphism $\pi_H:A\to A_H$ determined by 
$$\pi_H (x_{\sigma,i})=\begin{cases}
x_{\sigma,i} & \mbox{if $\s\in H$,} \\
0 & \mbox{otherwise}
\end{cases}$$
is a surjective morphism of Hopf C$^*$-algebras.
\end{pro}
\begin{proof}
Coassociativity of $\Delta_H$ is obvious. Since \eqref{unitarH} is satisfied, the matrix $(u_{ij})_{i\sim_Hj}$ is invertible in $M_n(A_H)$. Hence the density conditions in Definition \ref{cqg} are also satisfied. The rest follows from straightforward calculations.
\end{proof}
\begin{pro}
Let $\iota_H: C(H)\to A_H$ be a map given by $\iota_H(\delta_\s)=p_\s$ for $\s\in H$, and let $\pi_H:A_H\to C(\bT^n)$ be determined by $\pi_H(x_{\s,i})=x_i$ for $\s\in H$, $i=1,\ldots,n$ where $x_i$ are standard generators of $C(\bT^n)$. Then $\iota_H$ is an injective morphism of Hopf C*-algebras, while $\pi_H$ is a surjective one. 

Moreover,
$0\longrightarrow C(H)\stackrel{\iota_H}{\longrightarrow} A_H \stackrel{\pi_H}{\longrightarrow}  C(\bT^n)\longrightarrow 0$
is an exact sequence.
\end{pro}
\begin{proof}
Directly follows from \eqref{deltap} and straightforward calculations.
\end{proof}
\begin{m}
Application of the above proposition for $H=S_n$ yields the following short exact sequence of compact quantum groups
$$0\longrightarrow \bT^n \longrightarrow \bG_\te \longrightarrow S_n \longrightarrow 0.$$
In particular $\bT^n$ is a quantum normal subgroup of $(A_\theta^n,\Delta)$ in the sense of \cite{Wan14}, while $S_n$ is a quotient quantum group. Thus, the purely quantum group 
$\bG_\te$ can be considered as an extension of the classical group $\bT^n$ by the classical group $S_n$ (compare \cite{VW}).
\end{m}

Our aim is to describe irreducible representations of $\bG_\te$. Namely, we will show the following theorem.
\begin{thm}
\label{reps}
Irreducible representations of the quantum group $\bG_\te$ are in one-to-one correspondence with irreducible representations of the classical group $\bT^n\rtimes S_n$.
\end{thm}
\begin{m}
\label{remMac}
According to Mackey theory (\cite{Mac}), each irreducible representation of the semidirect product of the abelian group $\bT^n$ by $S_n$ can be characterized as an induced representation of certain representation of "little subgroup". For readers convenience let us remind basic ingredients of this construction. Consider the natural action of $S_n$ on $\bZ^n$ (which is a dual to $\bT^n$), fix some orbit $O_\kappa$ of this action and its representative $\vr_\kappa$. One can canonically extend the character $\vr_\kappa$ to a character of $\bT^n\times H_{\vr_\kappa}\subset \bT^n\rtimes S_n$. Then choose some irreducible representation $v$ of the stabilizer group $H_{\vr_\kappa}\subset S_n$, extend it to a representation $\tilde{v}$ of $\bT^n\times H_\vr\subset \bT^n\rtimes S_n$ and take a representation $\eta$ being a tensor product $\eta$ of the extended character $\vr_\kappa$ and $\tilde{v}$. The induced representation $\tilde{\eta}_{{\vr_\kappa},v}$ of $\eta$ turns out to be an irreducible representation of $\bT^n\rtimes S_n$. Two such representations $\tilde{\eta}_{\vr_\kappa,v}$ and $\tilde{\eta}_{\vr_{\kappa'},v'}$ are equivalent if and only if $\kappa=\kappa'$ and $v$, $v'$ are equivalent. Moreover, each irrediucible representation of $\bT^n\rtimes S_n$ is equivalent to $\tilde{\eta}_{\vr_\kappa,v}$ for some $\kappa$ and $v$.
\end{m}
Now, let us fix $\vr\in\bZ^n$. Let $H_{\vec{r}}\subset S_n$ be the stabilizer of $\vec{r}$, i.e. $\sigma\in H_{\vec{r}}$ if and only if $\sigma\vec{r}=\vec{r}$.
The immediate consequence of \eqref{deltau} is the following equality
\beq
\label{DeltaHu}
\Delta_{H_\vr}(u_\s^\vr)=\sum_{\t\in H_\vr}\tilde{\psi}_\vr(\t,\s)u_\t^\vr\otimes u_{\t^{-1}\s}^\vr
\eeq
for every $\s\in H_\vr$.

Let us consider the following function $\vartheta_\vr:H_\vr\times H_\vr\to\bC$ 
\beq
\vartheta_\vr(\s,\t)=\prod_{\substack{i<j \\\s(i)>\s(j)}}\left(\om_{\ti(j),\ti(i)}\right)^{r_{\s(i)}r_{\s(j)}},\qquad \s,\t\in S_n.
\eeq
We describe some properties of $\vartheta_\vr$.
\begin{lm}
\label{alfal}
For every $\s\in H_\vr$ and $\t,\rho\in S_n$,
\beq
\label{alfa_lem}
\vartheta_\vr(\s\t,\rho)=\vartheta_\vr(\s,\t\rho)\vartheta_\vr(\t,\rho).
\eeq
\end{lm}
\begin{proof}
The right hand side of \eqref{alfa_lem} is given by
\beg
\vartheta_\vr(\sigma,\tau\rho)\vartheta_\vr(\t,\rho)
&=&\prod_{\substack {i<j\\ \s(i)>\s(j)}}\om_{\rho^{-1}\ti(i),\rho^{-1}\ti(j)}^{r_{\s(i)}r_{\s(j)}}
\prod_{\substack {i<j\\ \t(i)>\t(j)}} \om_{\rho^{-1}(i),\rho^{-1}(j)}^{r_{\t(i)}r_{\t(j)}} 
\\&=&
\prod_{\substack {\t(j)<\t(i) \\ \s\t(j)>\s\t(i)}}
\omega_{\rho^{-1}(j),\rho^{-1}(i)}^{r_{\t(i)}r_{\t(j)}} 
\prod_{\substack {i<j \\ \t(i)>\t(j)}}
\omega_{\rho^{-1}(i),\rho^{-1}(j)}^{r_{\t(i)}r_{\t(j)}} 
\eeg
The last equality is due to substitution 
$i\mapsto\t(j)$, $j\mapsto\t(i)$ in the first product as well as the fact that $\s\in H_\vr$. Thus we get 
\beg
\vartheta_\vr(\sigma,\tau\rho)\vartheta_\vr(\t,\rho)&=&
\prod_{\substack {i<j \\ \t(i)>\t(j) \\ \s\t(i)<\s\t(j)}}\omega_{\rho^{-1}(j),\rho^{-1}(i)}^{r_{\t(i)}r_{\t(j)}} 
\prod_{\substack {i>j \\ \t(i)>\t(j) \\ \s\t(i)<\s\t(j)}}\omega_{\rho^{-1}(j),\rho^{-1}(i)}^{r_{\t(i)}r_{\t(j)}} \times \\
&& \qquad {} \times
\prod_{\substack {i<j \\ \t(i)>\t(j)\\ \s\t(i)<\s\t(j)}}\om_{\rho^{-1}(i),\rho^{-1}(j)}^{r_{\t(i)}r_{\t(j)}} 
\prod_{\substack {i<j \\ \t(i)>\t(j)\\ \sigma\t(i)>\sigma\t(j)}}\om_{\rho^{-1}(i),\rho^{-1}(j)}^{r_{\t(i)}r_{\t(j)}} .
\eeg
Since $\om_{ji}=\om_{ij}^{-1}$, it reduces to
$$
\vartheta_\vr(\sigma,\tau\rho)\vartheta_\vr(\t,\rho)=
\prod_{\substack {i>j \\ \t(i)>\t(j) \\ \s\t(i)<\s\t(j)}}\omega_{\rho^{-1}(j),\rho^{-1}(i)}^{r_{\t(i)}r_{\t(j)}}
\prod_{\substack {i<j \\ \t(i)>\t(j)\\ \sigma\t(i)>\sigma\t(j)}}\om_{\rho^{-1}(i),\rho^{-1}(j)}^{r_{\t(i)}r_{\t(j)}}.
$$  
After the substitution $i\mapsto j$, $j\mapsto i$ in the first product we get
\beg
\vartheta_\vr(\sigma,\tau\rho)\vartheta_\vr(\t,\rho)&=&
\prod_{\substack {i<j \\ \t(i)<\t(j) \\ \s\t(i)>\s\t(j)}}\omega_{\rho^{-1}(i),\rho^{-1}(j)}^{r_{\t(i)}r_{\t(j)}}
\prod_{\substack {i<j \\ \t(i)>\t(j)\\ \sigma\t(i)>\sigma\t(j)}}\om_{\rho^{-1}(i),\rho^{-1}(j)}^{r_{\t(i)}r_{\t(j)}} \\
&=&
\prod_{\substack {i<j \\ \s\t(i)>\s\t(j)}}\omega_{\rho^{-1}(i),\rho^{-1}(j)}^{r_{\t(i)}r_{\t(j)}} \\
&=&
\vartheta_\vr(\s\t,\rho) .
\eeg
Thus the proof of \eqref{alfa_lem} is completed.
\end{proof}

Observe that the previously defined function $\tilde{\psi}_\vr$ (cf. \eqref{psi-tilde}) satisfies 
\beq
\label{psi-alfa}
\tilde{\psi}_\vr(\s,\t)=\overline{\vartheta_\vr(\s,e)}\vartheta_\vr(\s,\si\t).
\eeq
Consider the following functions $\psi_\vr: S_n\times S_n\to \bT^1$ and $\varphi_\vr:H_\vr\to \bT^1$
\beq
\label{psii}
\psi_\vr(\s,\t)=\tilde{\psi}_\vr(\s,\s\t)=\overline{\vartheta_\vr(\s,e)}\vartheta_\vr(\s,\t), \quad \s,\t\in S_n,
\eeq
\beq
\label{varfi}
\varphi_\vr(\s)=\overline{\vartheta_\vr(\s,e)},\qquad \s\in H_\vr.
\eeq

Let us recall that for a locally compact group $G$ and abelian group $A$, a function $\psi:G\times G\to A$ is called a 2-cocycle, if $\psi(g,e)=\psi(e,g)=1$ for $g\in G$, and
\beq
\label{psi}
\psi(g,h)\psi(gh,k)=\psi(g,hk)\psi(h,k),\qquad g,h,k\in G.
\eeq

\begin{pro}
\label{kocykl}
For every $\sigma,\tau \in H_\vr$, we have
\begin{equation}
\label{psirep}
\varphi_\vr(\sigma)\varphi_\vr(\tau)=\varphi_\vr(\sigma\tau)\psi_\vr(\sigma,\tau).
\end{equation}

Consequently, the function $\psi_\vr$ given by \eqref{psii} is a 2-cocycle on $H_\vr$.
\end{pro}
\begin{proof}
Obvious consequence of definitions \eqref{psii}, \eqref{varfi} and relation \eqref{alfa_lem}.
\end{proof}

Define an element $\chi_\vr\in A_{H_\vr}$ by 
\beq
\label{chik}
\chi_\vr=\sum_{\s\in H_\vr}\varphi_\vr(\s)u_\s^\vr
\eeq
Then we have
\begin{pro}
For any $\vr\in\bZ^n$, $\Delta_{H_\vr}(\chi_\vr)=\chi_\vr\otimes \chi_\vr$, i.e. 
$\chi_\vr$ is a 1-dimensional representation (character) of a quantum group $(A_{H_\vr},\Delta_{H_\vr})$.
\end{pro}
\begin{proof}
From \eqref{DeltaHu} we have
\beg
\Delta_{H_\vr}(\chi_\vr) &=&
\sum_{\s\in H_\vr}\varphi_\vr(\s)\Delta_{H_\vr}(u_\s^\vr) \\
&=&
\sum_{\s,\t\in H_\vr}\varphi_\vr(\s)\tilde{\psi}_\vr(\t,\s)u_\t^\vr\otimes u_{\t^{-1}\s}^\vr  \\
&=&
\sum_{\t,\r\in H_\vr}\varphi_\vr(\t\r)\tilde{\psi}_\vr(\t,\t\r) u_\t^\vr \otimes u_\r^\vr \\
&=&
\sum_{\t,\r\in H_\vr}\varphi_\vr(\t\r)\psi_\vr(\t,\r) u_\t^\vr \otimes u_\r^\vr \\
&=&
\sum_{\t,\r\in H_\vr}\varphi_\vr(\t)\varphi_\vr(\r) u_\t^\vr \otimes u_\r^\vr \\
&=&
\chi_\vr \otimes \chi_\vr
\eeg
The fifth equality follows from Proposition \ref{kocykl}.
\end{proof}


Let us fix $\vr\in\bZ^n$. It can be considered as a character of the (classical) group $\bT^n$. Since $\bT^n$ is quantum subgroup of $(A_{H_\vr},\Delta_{H_\vr})$, the character $\chi_\vr$ defined in \eqref{chik} can be considered as an extension of $\vr$ to the whole $(A_{H_\vr},\Delta_{H_\vr})$. Now, let $v:H_\vr\to B(K)$ be some irreducible representation of the stabilizer subgroup $H_\vr$. Consider it as an element of $B(K)\otimes C(H_\vr)$, and let $\tilde{v}=(\id\otimes\iota_{H_\vr})(v)\in B(K)\otimes A_{H_\vr}$. It can be easily shown that $\tilde{v}$ is an irreducible representation of the quantum group $(A_{H_\vr},\Delta_{H_\vr})$.

Next, consider a representation $\eta\in B(K)\otimes A_{H_\vr} $ being a tensor product of representations $\chi_\vr$ and $\tilde{v}$. We will show that one can associate a representation $\tilde{\eta}_{\vr,v}$ to $\eta$ in a way similar to the construction of an induced representation.
To this end let us consider the subspace $\tilde{K}_{\vr,v}\subset K\otimes A_\te^n$ defined as follows
$$ \tilde{K}_{\vr,v}=\left\{ \sum_{\s\in S_n}\lambda(\s)\otimes u_\s^\vr:\,\lambda \in F_{H_\vr,v}\right\}
$$ 
where $F_{H_\vr,v}$ is a space of all functions $\lambda:S_n\to K$ such that $\lambda(\s\rho)=\vartheta_\vr(\s,\rho)v(\s)\lambda(\rho)$ for every $\s\in H_\vr$ and $\rho\in S_n$.
%
Now, for each right coset from $H_\vr\backslash S_n$ we fix some representative $\s_\nu$, so that 
$H_\vr \s_\nu\cap H_\vr \s_{\nu'}=\emptyset$ for $\nu\neq\nu'$ and $\bigcup_\nu H_\vr\s_\nu=S_n$.
We can define a basis in $F_{H_\vr,v}$ which consists of the functions
$ 
\lambda_{\mu,i}(\sigma_\nu)=\delta_{\mu,\nu}e_i,
$ 
where $\mu\in \left\{1,2,\ldots [G:H_\vr]\right\}$ and $e_i$ form a fixed basis in the space $K$. Since any $\sigma\in S_n$ belongs to a unique right coset $H_\vr\sigma_\nu$, one can write
$$ 
\lambda_{\mu,i}(\sigma)=\vartheta_\vr(\s\s_\nu^{-1},\s_\nu^{-1})v(\sigma\sigma_\nu^{-1})\delta_{\mu,\nu}e_i
$$ 
for all $\sigma\in S_n$. 
Let 
$$
f_{\mu,i}=\sum_{\s\in S_n}\lambda_{\mu,i}(\s)\otimes u_\s^\vr=\sum_{\s\in H_\vr\sigma_\mu}\lambda_{\mu,i}(\s)\otimes u_\s^\vr .
$$
The system of elements $f_{\mu,i}$ turns out to be a basis in $\tilde{K}_{\vr,v}$. 
If $K$ has the structure of a Hilbert space, then we can equipp $F_{H_\vr,v}$ with a scalar product
\begin{equation}\label{scalar}
\left\langle f,g\right\rangle_{F_{H_\vr,v}}=\frac{1}{|H_\vr|}\sum_{\pi\in S_n}\left\langle f(\pi),g(\pi)\right\rangle_K,
\end{equation}
One can check that $\left\{\lambda_{\mu,i}\right\}$ is orthonormal with respect to (\ref{scalar}).

Now, define a map $\tilde{\eta}_{\vr,v}:\tilde{K}_{\vr,v} \to  \tilde{K}_{\vr,v}\otimes A_\theta^n$ by 
$ 
\tilde{\eta}_{\vr,v}=\id_K\otimes\Delta.
$ 
\begin{lm}
\label{lema}
For any $\mu$ and $i$,
\begin{equation}
\label{etamat}
\tilde{\eta}_{\vr,v}(f_{\mu,i})=\sum_{\nu,j} f_{\nu,j}\otimes a_{\nu,j;\mu,i}
\end{equation}
where elements $a_{\nu,j;\mu,i}\in A_\te^n$ are given by
\begin{equation}
\label{amu}
a_{\nu,j;\mu,i}=\sum_{\xi\in H_\vr}\vartheta_\vr(\s_\nu,e) \overline{\vartheta_\vr(\xi^{-1}\t,\ti\xi\s_\mu)} \left\langle e_j,v(\xi)e_i\right\rangle u^{\sigma_\nu^{-1}\vr}_{\sigma_\nu^{-1}\xi\sigma_\mu}.
\end{equation}
\end{lm}
\begin{proof}
It follows from \eqref{deltau} that
\begin{eqnarray}
\tilde{\eta}_{\vr,v}(f_{\mu,i})&=&
\sum_{\s\in H_\vr\s_{\mu}}\sum_{\t\in S_n} \lambda_{\mu,i}(\s) \tilde{\psi}(\tau,\s) \otimes u_\tau^\vr \otimes u^{\ti\vr}_{\ti\s} \nonumber\\
&=&\sum_{\tau\in S_n} \sum_{\rho\in \ti H_\vr \s_{\mu}} \lambda_{\mu,i}(\tau\rho) \tilde{\psi}(\tau,\tau\rho) \otimes u_\tau^\vr \otimes u^{\ti\vr}_{\rho} \nonumber\\
&=&  \sum_{\nu} \sum_{\tau\in H_\vr\sigma_{\nu}} \sum_{\xi\in H_\vr}  \lambda_{\mu,i}(\xi\sigma_\mu) \tilde{\psi}_\vr(\t,\ti\xi\s_\mu) \otimes u_\tau^\vr\otimes u^{\sigma_\nu^{-1}\vr}_{\tau^{-1}\xi\sigma_\mu}\nonumber,
\end{eqnarray}
For a given coset $\mu$, and $\xi \in H_\vr$, and $\tau \in H_\vr\sigma_\nu$ we may define a function $\lambda_{\mu,i}^{\xi,\tau}:S_n\rightarrow K$ given by 
$$
\lambda_{\mu,i}^{\xi,\tau}(\pi)=\vartheta_\vr(\pi\ti,\t) \overline{\vartheta_\vr(\pi\ti,\xi\s_\mu)} \lambda_{\mu,i}(\pi\tau^{-1}\xi\sigma_\mu).
$$ 
We show that $\lambda_{\mu,i}^{\xi,\tau}\in F_{H_\vr,v}$.
If $\pi\not\in H_\vr\s_\nu$, then $\lambda_{\mu,i}^{\xi,\tau}(\s\pi)=0=\lambda_{\mu,i}^{\xi,\tau}(\pi)$ for all $\s\in H_\vr$. Let $\pi\in H_\vr\s_\nu$ and $\s\in H_\vr$. Observe that \eqref{alfa_lem} implies
\beg
\vartheta_\vr(\s\pi\ti,\t) \overline{\vartheta_\vr(\s\pi\ti,\xi\s_\mu)} \vartheta_\vr(\s,\pi\ti\xi\s_\mu)&=&
\vartheta_\vr(\s\pi\ti,\t)\overline{\vartheta_\vr(\pi\ti,\xi\s_\mu)} \\
&=&
\vartheta_\vr(\pi\ti,\t) \vartheta_\vr(\s,\pi) \overline{\vartheta_\vr(\pi\ti,\xi\s_\mu)}
\eeg
Therefore, both expressions
$$
\lambda_{\mu,i}^{\xi,\tau}(\s\pi)=
\vartheta_\vr(\s\pi\ti,\t) \overline{\vartheta_\vr(\s\pi\ti,\xi\s_\mu)} \vartheta_\vr(\s,\pi\ti\xi\s_\mu)v(\s) \lambda_{\mu,i}(\pi\ti\xi\s_\mu)
$$
and
$$
\vartheta_\vr(\s,\pi)v(\s)\lambda_{\mu,i}^{\xi,\tau}(\pi) =
\vartheta_\vr(\pi\ti,\t) \overline{\vartheta_\vr(\pi\ti,\xi\s_\mu)} \vartheta_\vr(\s,\pi) v(\s) \lambda_{\mu,i}(\pi\ti\xi\s_\mu)$$
are equal. Hence $\lambda_{\mu,i}^{\xi,\tau}\in F_{H_\vr,v}$. 
Consequently, $\lambda_{\mu,i}^{\xi,\tau}$ is a linear combination of functions $\lambda_{\beta,j}$, i.e. 
$ 
\lambda_{\mu,i}^{\xi,\tau}=\sum_{\beta,j}\gamma_{\beta,j}\lambda_{\beta,j},
$ 
where $\gamma_{\beta,j}$ are given by 
\begin{eqnarray}
\gamma_{\beta,j}&=&
\left\langle \lambda_{\beta,j},\lambda_{\mu,i}^{\xi,\tau}\right\rangle_{F_{H_\vr}} \\ 
&=&
\frac{1}{|H_\vr|} \sum_{\pi\in S_n} \left\langle \lambda_{\beta,j}(\pi),\lambda_{\mu,i}^{\xi,\tau}(\pi) \right\rangle \nonumber \\
&=& 
\frac{1}{|H_\vr|}\sum_{\pi\in S_n} \left\langle \lambda_{\beta,j}(\pi), \vartheta_\vr(\pi\ti,\t) \overline{\vartheta_\vr(\pi\ti,\xi\s_\mu)} \lambda_{\mu,i}^{\xi,\tau}(\pi\ti\xi\sigma_\mu) \right\rangle \nonumber\\
&=&  
\frac{\delta_{\beta\nu}}{|H_\vr|} \sum_{\pi\in H_\vr\sigma_{\nu}} \vartheta_\vr(\pi\ti,\t) \overline{\vartheta_\vr(\pi\ti,\xi\s_\mu)} \overline{\vartheta_\vr(\pi\s_\nu^{-1},\s_\nu)} \vartheta_\vr(\pi\ti\xi,\s_\mu) \times \nonumber \\
&& 
\qquad{} \times \left\langle v(\pi\sigma_\nu^{-1})e_j,v(\pi\tau^{-1}\xi)e_i \right\rangle \nonumber \\
&=&
\frac{\delta_{\beta\nu}}{|H_\vr|} \sum_{\pi\in H_\vr\sigma_{\nu}} \vartheta_\vr(\pi\ti,\t) \overline{\vartheta_\vr(\pi\s_\nu^{-1},\s_\nu)} \vartheta_\vr(\xi,\s_\mu) \langle e_j, v(\s_\nu^{-1}\ti\xi)e_i\rangle \nonumber \\
&=&  
\frac{\delta_{\beta\nu}}{|H_\vr|} \sum_{\zeta\in H_\vr} \vartheta_\vr(\zeta\s_\nu\ti,\t) \overline{\vartheta_\vr(\zeta,\s_\nu)} \vartheta_\vr(\xi,\s_\mu) \left\langle e_j,v(\sigma_\nu\tau^{-1}\xi)e_i\right\rangle \nonumber \\
&=&
\delta_{\beta\nu}\vartheta_\vr(\s_\nu\ti,\t)\vartheta_\vr(\xi,\s_\mu) \left\langle e_j,v(\sigma_\nu\tau^{-1}\xi)e_i\right\rangle . \nonumber
\end{eqnarray}
In the above equalities we used relation \eqref{alfa_lem} several times.

Since 
$\lambda_{\mu,i}(\xi\sigma_\mu)=\lambda_{\mu,i}^{\xi,\tau}(\tau)= \sum_{\beta,j}\gamma_{\beta,j}
\lambda_{\beta,j}(\tau)$ and \eqref{psi-alfa} is satisfied,
we have
\beg
\lefteqn{\tilde{\eta}_{\vr,v}(f_{\mu,i})=} \\
&=& 
\sum_{\nu} \sum_\beta  \sum_j \sum_{\tau\in H_\vr\sigma_{\nu}} \sum_{\xi\in H_\vr} \delta_{\beta\nu} \vartheta_\vr(\s_\nu\ti,\t) \vartheta_\vr(\xi,\s_\mu) \left\langle e_j,v(\sigma_\nu\tau^{-1}\xi)e_i\right\rangle \lambda_{\nu,j}(\tau) \times \\
&& \qquad {}\times \overline{\vartheta_\vr(\t,\ti\xi\s_\mu)} \vartheta_\vr(\t,e) \otimes u_\tau^\vr\otimes u^{\sigma_\nu^{-1}\vr}_{\tau^{-1}\xi\sigma_\mu}\nonumber
\eeg
Notice that according to Lemma \ref{alfal}, 
$$
\vartheta_\vr(\s_\nu\ti,\t)=\vartheta_\vr(\s_\nu,e)\overline{\vartheta_\vr(\t,e)}
$$
and
$$
\vartheta_\vr(\xi,\s_\mu)= \vartheta_\vr(\t,\ti\xi\s_\mu) \overline{\vartheta_\vr(\xi^{-1}\t,\ti\xi\s_\mu)},
$$
as $\s_\nu\ti,\xi\in H_\vr$. Consequently,
$$
\vartheta_\vr(\s_\nu\ti,\t) \vartheta_\vr(\xi,\s_\mu) \overline{\vartheta_\vr(\t,\ti\xi\s_\mu)} \vartheta_\vr(\t,e) = \vartheta_\vr(\s_\nu,e) \overline{\vartheta_\vr(\xi^{-1}\t,\ti\xi\s_\mu)}.
$$
Finally, we get
\beg
\lefteqn{\tilde{\eta}_{\vr,v}(f_{\mu,i})=}\\
&=& \sum_{\nu,j}\left( \sum_{\t\in H_\vr\s_\nu}\lambda_{\nu,j}(\t) \otimes u_\t^\vr \otimes \sum_{\xi\in H_\vr} \vartheta_\vr(\s_\nu,e) \overline{\vartheta_\vr(\xi^{-1}\t,\ti\xi\s_\mu)} \langle e_j,v(\s_nu\ti\xi)e_i\rangle u_{\ti\xi\s_\mu}^{\s_\nu^{-1}\vr}\right).
\eeg
Observe that the expression
$$
\sum_{\xi\in H_\vr} \vartheta_\vr(\s_\nu,e) \overline{\vartheta_\vr(\xi^{-1}\t,\ti\xi\s_\mu)} \langle e_j,v(\s_nu\ti\xi)e_i\rangle u_{\ti\xi\s_\mu}^{\s_\nu^{-1}\vr}
$$
is independent on the choice of $\tau\in H_\vr\sigma_\nu$ and it is equal to $a_{\nu,j;\mu,i}$ given by \eqref{amu}. Thus, \eqref{etamat} is proved.
\end{proof}

\begin{pro}
For every $\vr\in\bZ$ and an irreducible representation $v$ of $H_\vr$, the map $\tilde{\eta}_{\vr,v}$ is an irreducible representation of $\bG_\te$. 
\end{pro}
\begin{proof}
It follows from Lemma \ref{lema} that $\tilde{\eta}_{\vr,v}(\tilde{K}_{\vr,v})\subset \tilde{K}_{\vr,v}\otimes A_\theta^n$. Observe also that for any $f\in \tilde{K}_{\vr,v}$ we have
\beg
(\tilde{\eta}_{\vr,v}\otimes \id_{A_{\theta}^n})\tilde{\eta}_{\vr,v}(f) 
&=& 
(\id_K\otimes\Delta\otimes \id_{A_{\theta}^n})(\id_K\otimes\Delta)(f) \\
&=&
(\id_K\otimes \id_{A_{\theta}^n}\otimes \Delta)(\id_K\otimes\Delta)(f) \\
&=&
(\id_{\tilde{K}}\otimes \Delta)\tilde{\eta}_{\vr,v}(f),
\nonumber
\eeg
so $\tilde{\eta}_{\vr,v}$ is a representation.

We show that $\tilde{\eta}_{\vr,v}$ is irreducible.
Let $S:\tilde{K}_{\vr,v}\rightarrow \tilde{K}_{\vr,v}$ be a linear map which intervines $\tilde{\eta}_{\vr,v}$ with itself. Let $s^{\mu,i}_{\nu,j}$ be matrix coefficients of $S$ in the basis $f_{\nu,j}$, so that 
$$ 
S(f_{\mu,i})=\sum_{\nu,j}s^{\mu,i}_{\nu,j}f_{\nu,j}.
$$ 
Then \eqref{etamat} yields
$$ 
\tilde{\eta}_{\vr,v}(S(f_{\mu,i}))=\sum_{\nu,\pi=1}^{[S_n:H_\vr]}\sum_{j,k=1}^{\dim K}
s^{\mu,i}_{\nu,j}f_{\pi,k}\otimes a_{\pi,k;\nu,j}.
$$ 
and 
$$ 
(S\otimes \id_{A^n_{\theta}})\tilde{\eta}_{\vr,v}(f_{\mu,i})=\sum_{\beta,\pi}^{[S_n:H_\vr]}\sum_{k,l}^{\dim K}s^{\beta,l}_{\pi,k}f_{\pi,k}\otimes a_{\beta,l;\mu,i}.
$$ 
Since $S\in \Mor(\tilde{\eta},\tilde{\eta})$, the two above expressions are equal. 
Thus, 
for arbitrary $\pi$ and $k$
$$ 
\sum_{\nu,j}s^{\mu,i}_{\nu,j}a_{\pi,k;\nu,j}=\sum_{\beta,l}s^{\beta,l}_{\pi,k}a_{\beta,l;\mu,i}.
$$ 
Using the exact form \eqref{amu} of coefficients $a_{\pi,l;\mu,i}$ and taking into account that $u^{\sigma_\pi^{-1}k}_{\sigma_\pi^{-1}\xi\sigma_\mu}$ are lineary independent for different $\pi,\mu$ and $\xi$, we infer that for any $\xi\in H_k$,
$$ 
\sum_{j}s^{\mu,i}_{\nu,j}\left\langle e_k,v(\xi)e_j \right\rangle=0,
$$ 
provided that $\nu\neq\mu$, and
$$ 
\sum_{j}s^{\mu,i}_{\mu,j}\left\langle e_k,v(\xi)e_j \right\rangle=\sum_{l}s^{\pi,l}_{\pi,k}\left\langle e_l,v(\xi)e_i \right\rangle.
$$ 
These equalities are equivalent to 
\begin{equation}
\label{fin1}
\left\langle e_k,v(\xi)S_\nu^\mu e_i \right\rangle=0
\end{equation}
for $\nu\neq\mu$, and
\begin{equation}
\label{fin2}
\left\langle e_k,v(\xi)S_\mu^\mu e_i \right\rangle=\left\langle e_k,S_\pi^\pi v(\xi)e_i \right\rangle.
\end{equation}
where $S_\nu^\mu:K\rightarrow K$ is a linear operator with matrix entries $(s^{\mu,i}_{\nu,j})_{i,j}$ with respect to the standard basis of $K$. 
Equation (\ref{fin1}) means that $S_\nu^\mu=0$ for $\nu\neq\mu$, while 
(\ref{fin2}) shows that  $S_\mu^\mu\in \Mor(v,v)$ and $S_\mu^\mu=S_\pi^\pi$ for any $\mu,\pi$. It follows from irreducibility of $v$ that $s^{\mu,i}_{\nu,j}=\lambda\delta_{\mu,\nu}\delta_{i,j}$. Hence $S=\jed_{\tilde{K}_{\vr,v}}$, and irreducibility of $\tilde{\eta}_{\vr,v}$ follows.
\end{proof}

Now, we are ready to formulate the followin theorem.
\begin{thm}
Let $\vr_\kappa$, $\kappa\in\mathcal{K}$, be a system of representatives of orbits of the action of $S_n$ on $\bZ^n$.
\begin{enumerate}
\item
For any $\kappa,\kappa'\in\cK$ and representations $v$, $v'$, if $\tilde{\eta}_{\vr_\kappa,v}$,  $\tilde{\eta}_{\vr_{\kappa'},v'}$ are equivalent, then $\kappa=\kappa'$  and $v,v'$ are equivalent representations.
\item
Each irreducible representation of $\bG_\te$ is equivalent to $\tilde{\eta}_{\vr_\kappa,v}$ for some $\kappa$ and $v$.
\end{enumerate}
\end{thm}
\begin{proof}
(1)
Assume fistly that $\kappa\neq\kappa'$. Since $\vr_\kappa$ and $\vr_{\kappa'}$ are not in the same orbit, elements $u_\sigma^{\tau{\vr_\kappa}}$ and $u_\pi^{\rho\vr_{\kappa'}}$ are linearly independent for any $\sigma, \tau,\pi, \rho\in S_n$. Hence, there is no map $S:\tilde{K}_{\vr_\kappa,v}\rightarrow \tilde{K}_{\vr_{\kappa'},v'}$ such that $\tilde{\eta}_{\vr_{\kappa'},v'}S=(S\otimes \id_{A^n_\theta})\tilde{\eta}_{\vr_\kappa,v}$.
Now cosider the case $\kappa=\kappa'$ and $\tilde{v}$ and $\tilde{v'}$ are inequivalent. Assume that there exists an invertible $S\in \Mor(\tilde{\eta}_{\vr_\kappa,v},\tilde{\eta}_{\vr_{\kappa'},v'})$. Then clearly $\tilde{K}_{\vr_\kappa,v}$ and $\tilde{K}_{\vr_{\kappa'},v}$ are isomorphic. 
Let us identify $K$ and $K'$ and fix some orthonormal basis $\left\{e_i\right\}$. 
Let 
$$ 
S(f_{\mu,i})=\sum_{\nu,j}s^{\mu,i}_{\nu,j}f'_{\nu,j}.
$$ 
Having in mind the formula \eqref{amu} for matrix coefficients of $\tilde{\eta}_{\vr,v}$ and $\tilde{\eta}_{\vr',v'}$, 
we conculde that for any $\xi\in H_{\vr}$
$$ 
\left\langle e_k,v(\xi)S_\nu^\mu e_i \right\rangle=0
$$ 
if $\nu\neq\mu$, and
$$ 
\left\langle e_k,v'(\xi)S_\mu^\mu e_i \right\rangle=\left\langle e_k,S_\pi^\pi v(\xi)e_i \right\rangle,
$$ 
where $S_\nu^\mu:K\rightarrow K$ is a linear operator with matrix coefficients $(s^{\mu,i}_{\nu,j})_{i,j}$ with respect to the fixed basis. Observe that for any $\mu$, $S_\mu^\mu\in \Mor(v,v')$ and therefore all $S_\mu^\mu$ are not invertible because $v$ and $v'$ are not equivalent. Since $S_\nu^\mu=0$ for $\mu\neq \nu$,  $S:\tilde{K}\rightarrow \tilde{K'}$ can not be invertible.

(2)
For $\vr\in\bZ^n$, let $V$ be the left regular representation of $H_{\vr}$. It can be decomposed as a direct sum of irreducible finite dimensional representations
$$ 
V=\bigoplus_{\iota} v_\iota
$$ 
where $\iota$ runs over a finite set of indices. 
It is well known fact that the character of the regular representation $\chi_V:H_{\vr}\rightarrow \mathbb{C}$ given by $\chi_V(\sigma)=\mathrm{Tr}(V(\sigma))$ for all $\sigma \in H_{\vk}$ has the following properties
$$ 
\chi_V(\sigma)=\begin{cases} \left|H_\vr\right| &\mbox{if } \sigma=e \\
0 & \mbox{if } \sigma\neq e\end{cases}
$$ 
where $\left|H_\vr\right|$ denotes the rank of $H_{\vr}$ and
$ 
\chi_V=\sum_{\alpha}\chi_\iota
$, 
$\chi_\iota$ being a character of representation $v_\iota$.

Let $a_{\nu,j;\mu,i}^\iota$ be matrix coefficients of the representation $\tilde{\eta}_{\vr,v_\iota}$, i.e. according to \eqref{amu}
$$ 
 a_{\nu,j;\mu,i}^\iota=\sum_{\xi\in H_\vk} \vartheta_\vr(\s_\nu,e) \overline{\vartheta_\vr(\xi^{-1}\s_\nu,\s_\nu^{-1}\xi\s_\mu)} \left\langle e_j,v_\iota(\xi)e_i\right\rangle u^{\sigma_\nu^{-1}\vr}_{\sigma_\nu^{-1}\xi\sigma_\mu}.
$$ 
For any $\varrho\in H_\vr$ and $\mu,\nu\in H_\vr\backslash S_n$,
\beg 
\lefteqn{\sum_\iota \sum_{i,j}\left\langle e_i,v_\iota (\varrho^{-1})e_j\right\rangle a_{\nu,j;\mu,i}^\iota=}\\
&=&
\sum_\iota\sum_{i,j}\left\langle e_i,v_\iota(\varrho^{-1})e_j\right\rangle\sum_{\xi\in H_\vr}\vartheta_\vr(\s_\nu,e) \overline{\vartheta_\vr(\xi^{-1}\s_\nu,\s_\nu^{-1}\xi\s_\mu)} \left\langle e_j,v_\iota (\xi)e_i\right\rangle u^{\sigma_\nu^{-1}\vr}_{\sigma_\nu^{-1}\xi\sigma_\mu}\nonumber\\
&=&
\sum_{\xi\in H_\vr} \vartheta_\vr(\s_\nu,e) \overline{\vartheta_\vr(\xi^{-1}\s_\nu,\s_\nu^{-1}\xi\s_\mu)} \sum_\iota \left(\sum_{i,j}\left\langle e_i,v_\iota (\varrho^{-1})e_j\right\rangle\left\langle e_j,v_\iota (\xi)e_i\right\rangle \right) u^{\sigma_\nu^{-1}\vr}_{\sigma_\nu^{-1}\xi\sigma_\mu}\nonumber\\
&=& 
\sum_{\xi\in H_\vr}\vartheta_\vr(\s_\nu,e) \overline{\vartheta_\vr(\xi^{-1}\s_\nu,\s_\nu^{-1}\xi\s_\mu)} \sum_\iota \mathrm{Tr}(v_\iota(\varrho^{-1}\xi)) u^{\sigma_\nu^{-1}\vr}_{\sigma_\nu^{-1}\xi\sigma_\mu}\nonumber\\
&=& 
\sum_{\xi\in H_\vr}\vartheta_\vr(\s_\nu,e) \overline{\vartheta_\vr(\xi^{-1}\s_\nu,\s_\nu^{-1}\xi\s_\mu)} \mathrm{Tr}(V(\varrho^{-1}\xi))u^{\sigma_\nu^{-1}\vr}_{\sigma_\nu^{-1}\xi\sigma_\mu} \nonumber\\
&=& 
\sum_{\xi\in H_\vr}\vartheta_\vr(\s_\nu,e) \overline{\vartheta_\vr(\xi^{-1}\s_\nu,\s_\nu^{-1}\xi\s_\mu)} \left|H_\vr\right|\delta_{\xi,\varrho}u^{\sigma_\nu^{-1}\vr}_{\sigma_\nu^{-1}\xi\sigma_\mu}\\
&=&
\vartheta_\vr(\s_\nu,e) \overline{\vartheta_\vr(\xi^{-1}\s_\nu,\s_\nu^{-1}\xi\s_\mu)} |H_\vk|u^{\sigma_\nu^{-1}\vr}_{\sigma_\nu^{-1}\varrho\sigma_\mu}.
\eeg 
Hence
$$ 
u^{\sigma_\nu^{-1}\vr}_{\sigma_\nu^{-1}\varrho\sigma_\mu}= \left(\vartheta_\vr(\s_\nu,e) \overline{\vartheta_\vr(\xi^{-1}\s_\nu,\s_\nu^{-1}\xi\s_\mu)}\left|H_\vr\right|\right)^{-1} \sum_\iota\sum_{i,j}\left\langle e_i,v_\iota(\varrho^{-1})e_j\right\rangle a_{\nu,j;\mu,i}^\alpha.
$$ 
For any $\vr\in\bZ^n$ and $\sigma\in S_n$ there are unique $\kappa\in\cK$, $\mu,\nu\in H_\vr\backslash S_n$, and $\varrho\in H_\vr$ such that $\vr=\sigma_\nu^{-1}\vr_\kappa$ and $\sigma=\sigma_{\nu}^{-1}\varrho\sigma_\mu$.
Thus, each element $u^{\vr}_{\sigma}$ is a linear combination of matrix coefficients of representations $\tilde{\eta}_{\vr_\kappa,v}$. Consequently, these coefficients linearly span the unique dense Hopf $^*$-subalgebra $\mathcal{A}_\theta^n$ in $A_\theta^n$, and the proof is completed (see \cite{BCT05,Wor94} for details).
\end{proof}

\begin{proof}[Proof of Theorem \ref{reps}]
Immediate consequence of the above theorem and Remark \ref{remMac}.
\end{proof}

\section{Final remarks}

\begin{m}[Multiple noncommutative torus]
In \cite{DPS03} a multiple noncommutative torus $\mathcal{T}_n$ was considered. It turns out that $\mathcal{T}_n$ and $\cA_\te^n$ are isomorphic as *-algebras. Morever, coalgebra structures on $\mathcal{T}_n$ is isomorphic to $(\cA_\te^n,F\circ \Delta)$, where $F$ is the flip operator on $\cA_\te^n\otimes \cA_\te^n$ (cf. \eqref{deltau} and \cite[equation (5)]{DPS03}). 
\end{m}
\begin{m}[Quantum isometry groups]
In \cite{Bho,BhoG,Gos09} a quantum isometry group $\mathrm{QISO}(\mathds{T}_\theta^n)$ of noncommutative $n$-torus was considered. It is defined as the universal object in the category of compact quantum groups acting smoothly and isometrically on $\bT_\te^n$.
Isometric action is understood as an action commuting with a Laplacian $\mathcal{L}$  
defined by 
$$ 
\mathcal{L}(x^\vr)=-(r_1^2+r_2^2\ldots + r_n^2)x^\vr=-r^2x^\vr,
$$ 
where $x^{\vr}=x_1^{r_1}x_2^{r_2}\ldots x_n^{r_n}$, and $x_1,\ldots,x_n$ are generators of $C(\bT_\te^n)$.
It appears that $\bG_\te=(A_\theta^n,\Delta)$ is a quantum subgroup of $\mathrm{QISO}(\mathds{T}_\theta^n)$. 
To justify this statement it is enough to show that condition (b) of \cite[Definition 2.11]{Gos09} is satisfied, namely that $\alpha_\vartheta=(\id_{C(\mathds{T}_\theta^n)}\otimes \vartheta)\circ\alpha$ commutes with $\mathcal{L}$ on $\mathrm{Poly}(\bT_\te^n)$ for every state $\vartheta$ on $A_\te^n$.
We have
\begin{eqnarray}
\alpha_\vartheta\mathcal{L}(x^\vr)
&=&
(\id_{C(\mathds{T}_\theta^n)}\otimes \vartheta)\alpha\mathcal{L}(x^\vr)\nonumber
\\
&=&-r^2(\id_{C(\mathds{T}_\theta^n)}\otimes \vartheta)\alpha(x^\vr)
\nonumber\\
&=&-r^2\sum_{\tau}\vartheta\left(u^{r_1}_{\tau^{-1}(1),1}u^{r_2}_{\tau^{-1}(2),2}\ldots u^{r_n}_{\tau^{-1}(n),n}\right)x^{r_1}_{\tau^{-1}(1)}x^{r_2}_{\tau^{-1}(2)}\ldots x^{r_n}_{\tau^{-1}(n)} \nonumber\\
&=&-r^2\sum_{\tau}\vartheta(u^{\vr}_{\tau,e})\varphi_{\vr}(\tau)x^{\tau^{-1}\vr} ,\nonumber
\end{eqnarray}
and
\begin{eqnarray}
\mathcal{L}\alpha_\vartheta(x^\vr)
&=&
\mathcal{L}(\id_{C(\mathds{T}_\theta^n)}\otimes \vartheta)\alpha(x^\vr)\nonumber
\\
&=&\sum_{\tau}\vartheta\left(u^{r_1}_{\tau^{-1}(1),1}u^{r_2}_{\tau^{-1}(2),2}\ldots u^{r_n}_{\tau^{-1}(n),n}\right)\mathcal{L}\left(x^{r_1}_{\tau^{-1}(1)}x^{r_2}_{\tau^{-1}(2)}\ldots x^{r_n}_{\tau^{-1}(n)}\right)
\nonumber\\
&=&\sum_{\tau}\vartheta\left(u^{\vr}_{\tau,e}\right)\varphi_{\vr}(\tau)\mathcal{L}\left(x^{\tau^{-1}\vr}\right) \nonumber\\
&=&
{}-\sum_{\tau}(r_{\tau(1)}^2+r_{\tau(2)}^2\ldots r_{\tau(n)}^2)\vartheta\left(u^{\vr}_{\tau,e}\right)\varphi_{\vr}(\tau)x^{\tau^{-1}\vr} \nonumber\\
&=&{}-r^2\sum_{\tau}\vartheta\left(u^{\vr}_{\tau,e}\right)\varphi_{\vr}(\tau)x^{\tau^{-1}\vr} .\nonumber
\end{eqnarray}
Hence 
$ 
\alpha_\vartheta\mathcal{L}(x)=\mathcal{L}\alpha_\vartheta(x)
$ 
for all states $\vartheta$ on $A_\theta^n$ and all $x\in \mathrm{Poly}(\mathds{T}_\theta^n)$. This implies that $(\bG_\te,\alpha)$ is an object in the category of compact quantum groups acting on noncommutative $n$-torus in a smooth and isometric way \cite{Gos09}, because $\mathrm{Poly}(\mathds{T}_\theta^n)$ is a $^*$-algebra generated by eigenvectors of $\mathcal{L}$. Therefore, there is the unique quantum group morphism from $(\bG_\te,\alpha)$ to $(\mathrm{QISO}(\mathds{T}_\theta^n),\alpha_u)$ which is compatible with coactions $\alpha$ and $\alpha_u$. It is surjective on the level of underlying C*-algebras, since the coaction $\alpha$ is faithful. Thus $\bG_\te$ is a quantum subgroup of $\mathrm{QISO}(\bT_\te^n)$.
\end{m}
\begin{m}[Rieffel deformation]
Let us remind that the mentioned above quantum isometry group of the noncommutative torus can be viewed as an effect of the Rieffel deformation procedure applied to the classical isometry group of the classical torus (cf. \cite{Bho,BhoG}). It seems that the similar phenomena should appear in the context of quantum symmetry groups or, more generally, quantum groups preserving some distinguished set of subspaces. 
\end{m}

\section*{Acknowledgements}
M.B. was partially supported by the Fundation for Polish Science TEAM project "Technology for Information Transfer and Processing Based on Phenomena of a Strictly Quantum Nature" and by the University of Gda{\'n}sk grant BMN: Zad.538-5400-B294-16. We would like also to thank Andrzej Sitarz for stimulating discussions and Piotr So{\l}tan for pointing out the refference \cite{HM98}.

\end{document}